\documentclass{amsart}
\usepackage{amssymb,
enumitem,
mathrsfs,
mathtools,
tikz,
upgreek,
verbatim,
hyphenat,
url
}
\usepackage[T1]{fontenc}
\usepackage[shadow]{todonotes}
\usepackage{tikz-cd}

\usepackage{hyperref}
\definecolor{navyblue}{rgb}{0.0, 0.0, 0.5}
\hypersetup{colorlinks=true, allcolors = navyblue}

\newcommand{\setm}[2]{\{ #1 : #2 \}}

\newcommand{\Fix}{\mathrm{Fix}}

\DeclareMathOperator{\rank}{rank}

\DeclareMathOperator{\GL}{GL}
\DeclareMathOperator{\SL}{SL}

\DeclareMathOperator{\SO}{SO}
\DeclareMathOperator{\Spin}{Spin}
\DeclareMathOperator{\PSO}{PSO}
\DeclareMathOperator{\PGL}{PGL}
\DeclareMathOperator{\Fr}{Fr}
\DeclareMathOperator{\Rad}{Rad}
\DeclareMathOperator{\spn}{span}

\newcommand\N{\mathbb{N}}
\newcommand\Z{\mathbb{Z}}
\newcommand\Q{\mathbb{Q}}
\newcommand\R{\mathbb{R}}
\renewcommand\S{\mathbb{S}}

\newenvironment{enumerate-(a)}{\begin{enumerate}[label={\upshape (\alph*)}, leftmargin=2pc]}{\end{enumerate}}
\newenvironment{enumerate-(a)-r}{\begin{enumerate}[label={\upshape (\alph*)}, leftmargin=2pc,resume]}{\end{enumerate}}
\newenvironment{enumerate-(a)-5}{\begin{enumerate}[label={\upshape (\alph*)}, leftmargin=2pc,start=5]}{\end{enumerate}}
\newenvironment{enumerate-(A)}{\begin{enumerate}[label={\upshape (\Alph*)}, leftmargin=2pc]}{\end{enumerate}}
\newenvironment{enumerate-(A)-r}{\begin{enumerate}[label={\upshape (\Alph*)}, leftmargin=2pc,resume]}{\end{enumerate}}
\newenvironment{enumerate-(i)}{\begin{enumerate}[label={\upshape (\roman*)}, leftmargin=2pc]}{\end{enumerate}}
\newenvironment{enumerate-(i)-r}{\begin{enumerate}[label={\upshape (\roman*)}, leftmargin=2pc,resume]}{\end{enumerate}}
\newenvironment{enumerate-(I)}{\begin{enumerate}[label={\upshape (\Roman*)}, leftmargin=2pc]}{\end{enumerate}}
\newenvironment{enumerate-(I)-r}{\begin{enumerate}[label={\upshape (\Roman*)}, leftmargin=2pc,resume]}{\end{enumerate}}
\newenvironment{enumerate-(1)}{\begin{enumerate}[label={\upshape (\arabic*)}, leftmargin=2pc]}{\end{enumerate}}
\newenvironment{enumerate-(1)-r}{\begin{enumerate}[label={\upshape (\arabic*)}, leftmargin=2pc,resume]}{\end{enumerate}}
\newenvironment{itemizenew}{\begin{itemize}[leftmargin=2pc]}{\end{itemize}}

\newtheorem{theorem}{Theorem}[section]
\newtheorem{lemma}[theorem]{Lemma}
\newtheorem{corollary}[theorem]{Corollary}
\newtheorem{proposition}[theorem]{Proposition}

\theoremstyle{definition}
\newtheorem{definition}[theorem]{Definition}

\newtheorem{example}[theorem]{Example}

\theoremstyle{remark}
\newtheorem{remark}[theorem]{Remark}

\begin{document}

\title[Rotation equivalence and cocycle superrigidity]{Rotation equivalence and cocycle superrigidity}

\date{}
\author[F.~Calderoni]{Filippo Calderoni}

\address{Department of Mathematics, Rutgers University, 
Hill Center for the Mathematical Sciences,
110 Frelinghuysen Rd.,
Piscataway, NJ 08854-8019}
\email{filippo.calderoni@rutgers.edu}

 \subjclass[2010]{Primary: 03E15, 37A20, 54H05}
 \keywords{countable Borel equivalence relations, Borel reducibility, cocycle superrigidity, $S$-arithmetic groups}
\thanks{We would like to thank Alekos Kechris for pointing out Theorem~\ref{thm : superrigidity} and its relevance to this project. Also, we would like to thank Dave Morris and Alex Furman for some insightful discussions about Margulis superrigidity theorem.  Last we thank the anonymous referee for their careful reading and valuable comments.
}

\maketitle

\begin{abstract}
We analyze Euclidean spheres in higher dimensions and the corresponding orbit equivalence relations induced by the group of rational rotations from the viewpoint of descriptive set theory. It turns out that such equivalence relations are not treeable in dimension greater than $2$. Then we
show that the rotation equivalence relation in dimension \(n \geq 5\) is not Borel reducible to the one in any lower dimension. Our methods combine a cocycle superrigidity result from the works of Furman and Ioana with the superrigidity theorem for \(S\)-arithmetic groups of Margulis.
We also apply our techniques to give a geometric proof of the existence of uncountably many pairwise incomparable equivalence relations up to Borel reducibility. 
\end{abstract}


\section{Introduction}
An equivalence relation \(E\) on the standard Borel space \(X\) is said to be \emph{Borel} if \(E \subseteq X\times X\) is a Borel subset of \(X\times X\). A Borel equivalence relation \(E\) is said to be \emph{countable} if every \(E\)-equivalence class is countable. If \(E\) and \(F\) are countable Borel equivalence relations on the standard Borel spaces \(X\) and \(Y\), respectively, then a Borel map \(f \colon X \to Y\) is said to be a \emph{homomorphism} from \(E\) to \(F\) if for all \(x, y \in X\),
\[
x\mathbin{E}y\implies f(x)\mathbin{F}f(y).
\]
A \emph{weak Borel reduction} from \(E\) to \(F\) is a countable-to-one Borel homomorphism.
Whenever \(f\) satisfies the stronger property that for all \(x, y \in X\),
\[
x\mathbin{E}y \iff f(x)\mathbin{F}f(y),
\]
then
we say that
\(E\) is \emph{Borel reducible} to \(F\), that
\(f\) is a \emph{Borel reduction}, and we write \(E\leq_{B}F\).

Most of the Borel equivalence relations that we will consider in this paper arise from group actions as follows. Let \(\Gamma\) be a countable discrete group. Then a \emph{standard Borel \(\Gamma\)-space} is a standard Borel space \(X\) equipped with a Borel action \(\Gamma\times X\to X, (g, x) \mapsto g \cdot x\) of \(\Gamma\) on \(X\). We denote by \(\mathcal{R}(\Gamma\curvearrowright X)\) the corresponding \emph{orbit equivalence relation} on \(X\), whose classes are the \(\Gamma\)-orbits.  That is,
\[
\mathcal{R}(\Gamma\curvearrowright X) \coloneqq \{(x,y)\in X^2
\mid \exists g\in \Gamma\, (g\cdot x = y)\}.
\]
\noindent
Whenever \(\Gamma\) is a countable group, it is clear that \(\mathcal{R}(\Gamma\curvearrowright X)\) is a countable Borel equivalence relation. A classical result of Feldman and Moore~\cite{FelMoo}
establishes that all countable Bore equivalence relations arise in this manner. Precisely, if \(E\) is an arbitrary countable Borel equivalence relation on some standard Borel space \(X\), then there exist a countable group \(\Gamma\) and a Borel action of \(\Gamma\curvearrowright X\) such that \(E = \mathcal{R}(\Gamma\curvearrowright X)\).

The structure of the class of countable Borel equivalence relations up to Borel reducibility has been studied intesively over the last few decades.
By a classical dichotomy by Silver~\cite{Sil80} the collection of all (countable) Borel equivalence relations with uncountably many classes has a minimum element, the identity relation on real numbers, denoted by \(=_\R\). A Borel equivalence relation is called \emph{concretely classifiable} (or \emph{smooth}) if it is Borel reducible to \(=_\R\).

An example of countable Borel equivalence relation that is not concretly classifiable is the relation of eventual equality \(E_0\) on \(2^\N\), the Cantor space of binary sequences equipped with the product topology.
In fact,  Harrington, Kechris and Louveau~\cite{HarKecLou} showed that if \(E\) is a (countable) Borel equivalence relation and \(E\) is not smooth, then \(E_0\leq_B E\). Therefore,  \(E_0\) is an immediate successor of \(=_\R\) up to Borel reducibility.

At the other extreme we have the phenomenon of universality.
A countable Borel equivlence  relation \(F\) is said to be \emph{universal} if and only if \(E\leq_B F\) for every countable Borel equivalence relation \(E\).
The \emph{universal countable Borel equivalence relation} \(E_\infty\) is clearly unique up to Borel bi-reducibility, and has many natural realizations throughout mathematics.  For instance,  the orbit relation induced by shift action of \(\mathbb{F}_2\) on \(2^{\mathbb{F}_2}\) is universal.  Other examples of universal coutable Borel equivalence relations were found in~\cite{CalCla,DouJacKec, Gao00,Gao01,ThoVel99,ThoVel01}. 

While \(=_\mathbb{R}\), \(E_0\),  and \(E_\infty\) are easily seen to be linearly ordered by \(\leq_B\),  the structure of the interval \([E_0, E_\infty]\) is far from settled. 
A technical obstacle to separate \(E_0\) from other nonsmooth Borel equivalence relations was noted first by Hjorth and Kechris~\cite{HjoKec}, who showed that every countable Borel equivalence relation is Borel reducible to \(E_0\) when restricted on a comeager subset. 
Therefore, 
showing that a countable Borel equivalence relation lies strictly between \(E_0\) and \(E_\infty\) is ``generically'' hard,  and  descriptive theoretical methods alone are not enough.
The subject fluorished after the groundbreaking work of Adams and Kechris~\cite{AdaKec} who showed that the structure of countable Borel equivalence relations under Borel reducibility is extremely rich.  In fact,  there are uncountable many pairwise incomparable countable Borel equivalence relations up to Borel reducibility. Since then more results of incomparability have been discovered in the work of Thomas~\cite{Tho03APAL,Tho11, Tho12},  Hjorth and Kechris~\cite{Hjo12, HjoKec05},  and Coskey~\cite{Cos10,Cos12}.

In this paper, we use the theory of countable Borel equivalence relation to analyze the equivalence relation induced by countable groups of rotations on the spheres in higher dimensions.  A main motivation is continuing Zimmer's long-standing program of describing Lie groups and lattice actions on manifolds from the different viewpoint of descriptive set theory. 
Moreover,  our results uncover new examples of set theoretic rigidity (explained below) and are poised to shed light on the structure of the interval \([E_0, E_\infty]\). 

For \(n>1\), let \(\mathbb{S}^{n-1}\) be the (hyper)sphere in the \(n\)-dimensional Euclidean space.
As usual, we define
\(\mathbb{S}^{n-1} =\{x\in \mathbb{R}^n : ||x||=1\} \).
Since the linear transormations of \(\SO_n(\mathbb{R})\) preserves the dot product of \(\mathbb{R}^n\),  we let \(\SO_n(\mathbb{R})\) act on \(\mathbb{S}^{n-1}\) for all \(n>1\).
More specifically,  we consider the action of the countable subgroup of rational rotations \(\SO_n(\mathbb{Q})\) on \(\mathbb{S}^{n-1}\) and the induced orbit equivalence relation.  To simplify our notation,  we let 
\(\mathcal{R}_n\coloneqq \mathcal{R}(\SO_n(\mathbb{Q})\curvearrowright \mathbb{S}^{n-1})\). Since \(\SO_2(\mathbb{Q})\) is abelian, \(\mathcal{R}_2\) is easily seen to be reducible to \(E_0\).  (Every orbit equivalence relation induced by the action of a countable abelian group is necessarily hyperfinite by a result of Gao and Jackson~\cite{GaoJac}.) Moreover,  a simple argument of generic ergodicity shows that \(\mathcal{R}_2\) is not concretely classifiable,  thus it is Borel bi-reducible with \(E_0\).  Then, it is natural to investigate the complexity of \(\mathcal{R}_n\) in higher dimension.  Our discussion in Section 3 and 4 shows that for no \(n> 2\),  the orbit equivalence relation \(\mathcal{R}_n\) is Borel reducible to \(E_0\). In fact, we can derive the following stronger statement. (See Proposition~\ref{thm : antitreable >4} and Corollary~\ref{cor : antitreeable 3,4}.)

\begin{theorem}
\label{thm : not treeable}
\(\mathcal{R}_n\) is not treeable for all \(n\geq 3\). 
\end{theorem}

To see why Theorem~\ref{thm : not treeable} implies that no \(\mathcal{R}_n\) is hyperfinite for \(n>2\), we recall that the class of treeable countable Borel equivalence relations is downword closed under Borel reducibility.  That is,  if \(E\) is Borel reducible to \(F\) and \(F\) is treeable, then \(E\) is treeable too.  Moreover,  it is well-known that \(E_0\) is treeable.

In light of the results of Adams and Kechris~\cite{AdaKec} and Thomas~\cite{Tho03} it seems natural to conjecture that the complexity of \(\mathcal{R}_n\) strictly increases with the dimension. While we still do not know whether \(\mathcal{R}_n\) is Borel reducible to \(\mathcal{R}_{n+1}\) for \(n>2\), we prove the following main result:

\begin{theorem}
\label{thm : main}
If \(n\geq 5\) and  \(n>m\),  then  \(\mathcal{R}_n\) is not Borel reducible to \(\mathcal{R}_m\). 
\end{theorem}

As an immediate consequence,  we obtain that all \(\mathcal{R}_{n}\) in dimension three and above are of intermediate Borel complexity.

\begin{corollary}
For all \(n\geq 2\),  the countable Borel equivalence relation \(\mathcal{R}_n\) is not universal.
\end{corollary}

This also gives a new instance of \emph{set theoretic rigidity}.  That is, the phenomenon when the quotient space of a certain orbit equivalence relation encodes information about the action and the acting group, and such information can be recovered by analyzing simply the Borel complexity.   Precisely,  for \(\{m,n\}\neq\{3,4\}\), our result shows that \(\mathcal{R}_m\) and \(\mathcal{R}_n\) are bi-reducible if and only if \(m=n\),  and thus \(\SO_m(\mathbb{Q}) = \SO_n(\mathbb{Q})\).

To prove Theorem~\ref{thm : main} we first concentrate on the free part of the action \(\SO_n(\mathbb{Q})\curvearrowright \mathbb{S}^{n-1}\).
When \(X\) is a standard Borel
 \(\Gamma\)-space denote by \(\Fr_\Gamma X\) the \emph{free part} of the action \(\Gamma\curvearrowright X\). That is, 
\(\Fr_\Gamma X\coloneqq \{x \in X: \forall g \neq 1_\Gamma(g \cdot x\neq x)\}\). Whenever \(\Gamma\) is clear from the context we let \(\Fr X = \Fr_\Gamma X\).
Since \(\Fr X\) is a \(\Gamma\)-invariant Borel set, we denote by \(\mathcal{R}^*_n\)  the restriction of \(\mathcal{R}_n\) to the free part. That is, \(\mathcal{R}^*_n \coloneqq \mathcal{R}(\SO_n(\Q)\curvearrowright \Fr \mathbb{S}^{n-1})\).
In Section~5 we reach the following stepping stone towards Theorem~\ref{thm : main}.

\begin{theorem}
\label{thm : irreducibility free}
If \(n\geq 5\) and \(m< n\),  then \(\mathcal{R}_n\) is not Borel reducible to \(\mathcal{R}^*_m\).  In particular,  \(\mathcal{R}_n^*\) is not Borel reducible to \(\mathcal{R}^*_m\).
\end{theorem}

For dimension three we have a better understanding of the nonfree part of the action \(\SO_m(\Q)\curvearrowright \mathbb{S}^{m-1}\).  So, using Theorem~\ref{thm : irreducibility free}, we can derive Theorem~\ref{thm : main} for all \(n\geq 5\) and \(m=3\) directly. However,  analyzing the non-free part of the action of \(\SO_m(\mathbb{Q})\) for \(m\geq 4\) is much more subtle and accounts
for most of Section~5.

Our proofs combine techniques from ergodic theory such as the analysis of cocycles associated to Borel homomorphisms and free actions,  Kazhdan property \((\mathrm{T})\), and Margulis' superrigidity results~\cite{Mar}. The main tool is a cocycle rigidity result for compact actions which follows directly from the work of Furman~\cite{Fur11} and Ioana~\cite{Ioa16}.  It should be pointed out that their results cover only groups with Kazhdan' s property \(\mathrm{(T)}\). Initially, this might seem the main difficulty to deal with the sphere of dimension four because \(\SO_4(\mathbb{Q})\) does not contain any subgroup with such property.
However,  Drimbe and Vaes~\cite{DriVae} recently extended the results of cocycle superrigidity for the actions of dense subgroups of Lie groups on homogeneous spaces  without relying on property \(\mathrm{(T)}\).   Nevertheless, the case \(n=4\) is more subtle for purely algebraic reasons and it remains open whether \(\mathcal{R}_4\) is Borel reducible to \(\mathcal{R}_3\).

In the last section we discuss some further application of our methods.  We prove the existence of continuum many pairwise incomparable subrelations of \(\mathcal{R}_n\), for any \(n\geq 5\).

\section{Preliminaries}

\subsection{Ergodic theory and descriptive set theory}

Suppose that \(\Gamma\) is a countable group and that \(X\) is a standard Borel \(\Gamma\)-space. If \(\mu\) is a \(\Gamma\)-invariant probability measure on \(X\), then the action of \(\Gamma\) on \((X, \mu)\) is said to be \emph{ergodic} if for every \(\Gamma\)-invariant Borel subset \(A \subseteq X\), either \(\mu(A) = 0\) or \(\mu(A) = 1\).

\begin{theorem}
Let \(\mu\) be a \(\Gamma\)-invariant probability measure on the standard Borel \(\Gamma\)-space \(X\), then the following statements are equivalent:
\begin{enumerate-(i)}
\item
The action of \(\Gamma\) on \((X, \mu)\) is ergodic.
\item
If \(Y\) is a standard Borel space and \(f \colon X \to Y\) is a \(\Gamma\)-invariant Borel function, then there is
a \(\Gamma\)-invariant Borel subset \(M \subseteq X\) with \(\mu(M) = 1\) such that \(f\restriction M\) is a constant function.
\end{enumerate-(i)}
\end{theorem}

When \(\Gamma\) is a countable discrete group and \(\Lambda<\Gamma\) is a subgroup, the following are equivalent.

\begin{enumerate-(1)}
\item
\([\Gamma: \Lambda]<\infty\).
\item
\label{component : 2}
There exists a constant \(c>0\) such that for any measure-preserving ergodic action of \(\Gamma\) on a standard probability space \((X,\mu)\) there is a measurable \(\Lambda\)-invariant subset \(Z\subseteq X\) with \(\mu(Z)\geq c\) and \(\Lambda\) acts ergodically on \((Z, \mu_{Z})\), where \(\mu_{Z}\) is the probability measure defined by \(\mu_{Z}(A) = \mu(A)/\mu(Z)\).
\end{enumerate-(1)}

When \(Z\) is as in clause~\ref{component : 2} we say that \(Z\) is an \emph{ergodic component} for the action of \(\Lambda\) on \(X\).

Suppose that \(\Gamma\) is a countable group and that \(X\) is a standard Borel \(\Gamma\)-space with an invariant probability measure \(\mu\). Until further notice let \(E=\mathcal{R}(\Gamma\curvearrowright X)\) and \(F\) be a Borel equivalence relation on a standard Borel space \(Y\). Then \((E,\mu)\) is said to be \emph{\(F\)-ergodic} if for every Borel homomorphism \(f \colon X \to Y\) from \(E\) to \(F\), there exists a \(\Gamma\)-invariant Borel subset \(M \subseteq X\) with \(\mu(M ) = 1\) such that \(f\) maps \(M\) into a single \(F\)-class. In this case, for simplicity, we will usually say that \(E\) is \(F\)-ergodic.

Notice that ergodicity is a strong obstruction to Borel reducibility. When the action of \(\Gamma\) on \((X,\mu)\) is ergodic, then \(E\) is not Borel reducible to \(=_{\R}\). Moreover, provided that \(E\) is \(F\)-ergodic, then \(E\) is not weakly Borel reducible to \(F\). In particular, \(E\) is not Borel reducible to \(F\).

Recall that an equivalence relation relation \(F\) is said to be \emph{finite} if every \(F\)-equivalence class is finite.
A countable Borel equivalence relation \(E\) on the standard Borel space \(X\) is said to be hyperfinite if there exists an increasing sequence
\[F_{0} \subseteq F_{1} \subseteq\dotsb \subseteq F_{n} \subseteq\dotsb\]
of finite Borel equivalence relations on \(X\) such that \(E = \bigcup_{n\in \N} F_{n}\). 

Next proposition is the main technique to prove that an orbit equivalence relation is not hyperfinite. First, recall that a standard Borel \(\Gamma\)-space is said to be \emph{free} if the associated \(\Gamma\)-action is free.

\begin{proposition}
\label{prop : amenable}
Let \(\Gamma\) be a countable nonamenable group,  let \(X\) be a free standard Borel \(\Gamma\)-space and \(\mu\) be a \(\Gamma\)-invariant probability measure on \(X\). Then
\(\mathcal{R}(\Gamma\curvearrowright X)\) is not hyperfinite.
\end{proposition}

A long-standing open problem in descriptive set theory asks whether
any orbit equivalence relation generated by a Borel
action of a countable amenable group is hyperfinite. Below we will mention the following important result by Gao and Jackson.

\begin{theorem}[Gao--Jackson~\cite{GaoJac}]
\label{thm : GaoJac}
If \(\Gamma\) is a countable abelian group and \(\Gamma\) acts on a standard Borel  \(X\) in a Borel fashion, then \(\mathcal{R}(\Gamma\curvearrowright X)\) is hyperfinite.
\end{theorem}

A \emph{graphing} of an equivalence relation is a graph whose connected
components coincide with the equivalence classes.  We say that a countable
Borel equivalence relation \(E\) on a standard Borel space is \emph{treeable} if
there is a Borel acyclic graphing of \(E\). 
If \(\mathbb{F}_n\) is the free group on \(n\) generators for some \(n\in \mathbb{N}\cup\{\infty\}\),  and \(\mathbb{F}_n\) acts freely on the standard Borel space \(X\), then \(\mathcal{R}(\mathbb{F}_n\curvearrowright X)\) is treeable. 

The class of treeable countable Borel equivalence relations admits a universal element, which si defined as the orbit equivalence relation \(\mathcal{R}(\mathbb{F}_2\curvearrowright \Fr 2^{\mathbb{F}_2})\) induced by the shift action of \(\mathbb{F}_2\) on the free part of \(2^{\mathbb{F}_2}\coloneqq\big\{x\mid x\colon \mathbb{F}_2\to \{0,1\}\big\}\).
Clearly,  \(\mathcal{R}(\mathbb{F}_2\curvearrowright \Fr 2^{\mathbb{F}_2})\) is not hyperfinite for Proposition~\ref{prop : amenable}.

The following are some well-known properties of hyperfinite and treeable equivalence relations. (E.g., see~\cite[Proposition~1.3--3.3]{JacKecLou}.) An immediate consequence is that every hyperfinite equivalence relation is treeable.

\begin{proposition}
\label{prop : containment}
Let \(E\) and \(F\) be countable Borel equivalence relations.
\begin{enumerate-(a)}
\item
\label{item : a}
If \(E \subseteq F\) and \(F\) is hyperfinite (respectively treeable), then \(E\) is hyperfinite (respectively treeable).
\item
\label{item : b}
If \(E \leq_B F\) and \(F\) is hyperfinite (respectively treeable),  then \(E\) is hyperfinite (respectively treeable).
\end{enumerate-(a)}
\end{proposition}

\subsection{Algebraic groups and lattices}

Let \(\Omega\) to be a fixed algebraically closed field of
characteristic \(0\) containing \(\mathbb{R}\) and all \(p\)-adic fields \(\mathbb{Q}_p\),  with \(p\) prime. 

An \emph{algebraic group} \(G\) is a subgroup of the general linear group \(\GL_n(\Omega)\) which is Zariski closed in \(\GL_n (\Omega)\),  i.e.,  \(G\) consists of all matrices \(M\) in
\(\GL_n(\Omega)\) which satisfy a set of equations \(f_1(M) = 0, \dotsc, f_k(M) = 0\), where each
\(f_i\) is a polynomial in \(\Omega[x_1, \dotsc, x_n]\).  When the equations defining \(G\) have coefficients in a subfield \(k \subseteq \Omega\),  we call \(G\) a
\(k\)-group. 

If \(R\) is a subring of \( \Omega\),  we let
\[
\GL_n(R) = \{(a_{ij}) \in \GL_n(\Omega): a_{ij} \in R\text{ and }(\det(a_{ij}))^{-1} \in R\},
\]
and define for any algebraic group \(G \subseteq \GL_n(\Omega)\),
\[
G(R) = G\cap \GL_n(R).
\]
In particular,  if \(G\) is a \(k\)-group and \(R = k\),  then \(G(k)\) is the group of all matrices in \(G\) with coefficients in \(k\). Typical examples include the \emph{special linear} group \(\SL_n(k)\) and the \emph{special orthogonal group} \(\SO_n(k)\). Moreover, when \(p\) is a prime number we denote by \(\mathbb{Z}\big[\frac{1}{p}\big]\) the ring generated by \(\frac{1}{p}\). Then, \(\SO_n(\mathbb{Z}\big[\frac{1}{p}\big])\) denotes the groups of special orthogonal matrices in \(\SO_n(\mathbb{R})\) with rational entries, whose denominators are powers of \(p\).

If \(G\) is an algebraic \(k\)-group, then \(G\) is said to be \emph{\(k\)-simple} if every proper normal \(k\)-subgroup is trivial and \emph{almost \(k\)-simple} if every proper normal \(k\)-subgroup is finite.

Suppose that \(G\leq \GL_n(\Omega)\) is an algebraic \(k\)-group, where \(k\) is either \(\mathbb{R}\) or the field \(\mathbb{Q}_p\) of \(p\)-adic numbers for some prime \(p\).  In this case,  \(G(k)\leq\GL_n(k)\) is a lcsc group with respect to the \emph{Hausdorff  topology}; i.e., the topology obtained by restricting the natural topology on the product space \(k^{n^2}\) to \(G(k)\). Unless otherwise specified any topological notion about \(G(k)\) will refer to the Hausdorff topology.  

Let \(G\) be a connected algebraic group.  The \emph{(solvable) radical} \(\Rad(G)\) is the maximal normal connected solvable subgroup of \(G\).
The connected algebraic group \(G\) is called \emph{semisimple} if and only if
 \(\Rad(G)\) is trivial.  Algebraic groups considered in this paper are generally semisimple. When \(G\) is a semisimple algebraic group defined over \(k\), then the \(k\)-\emph{rank} of \(G\), in symbols \(\rank_k(G)\), is defined as the maximal dimension of 
an abelian \(k\)-subgroup of \(G\) which is \(k\)-split,  i.e.,  which can be diagonalized over \(k\).  Groups with \(\rank_k > 0\) (respectively \(\rank_k = 0\)) are called \emph{\(k\)-isotropic} (respectively \emph{\(k\)-anisotropic}).

Recall that if \(G\) is a locally compact second countable group and
\(\Gamma\leq G\) is a subgroup, 
we say that \(\Gamma\) is a \emph{lattice} in \(G\) if and only if \(\Gamma\) is discrete and there is an invariant Borel probability measure for the canonical action of \(G\) on \(G/\Gamma\) defined by \((g, h\Gamma)\mapsto gh\Gamma\).

Useful examples of lattices arise in the context of arithmetic groups. (See also \cite[Chapter~10]{Zim}.) We briefly describe a few examples that will appear later in this paper.

Let \(S = \{p_1, \dotsc, p_t\}\) be a finite nonempty set of primes.  
For every \(p\in S\),  we denote by \(\mathbb{Q}_p\) the field of \(p\)-adic numbers.
When \(G\) is an algebraic \(\Q\)-group, we can identify
\(\Gamma_S \coloneqq G\big(\mathbb{Z}[\frac{1}{p_1}, \dotsc, \frac{1}{p_n}]\big)\) with its image under the diagonal embedding into
\[
G(\mathbb{R}) \times G(\mathbb{Q}_{p_1} ) \times \dotsb \times G(\mathbb{Q}_{p_t}).\]
The group \(\Gamma_S\) is a typical example of \(S\)-arithmetic group. We shall define \(S\)-arithmetic groups and discuss some relevant properties in Section~5.2.
Here we recall the following result.

\begin{theorem}[A. Borel~\cite{Bor63},]
If \(G\) is a connected semisimple \(\Q\)-group,  then \(\Gamma_S\) is a lattice in \(G(\R)\times G(\mathbb{Q}_{p_1} ) \times \dotsb \times G(\mathbb{Q}_{p_t})\). 
\end{theorem}

In particular,  we shall use the fact that \(\SO_n\big(\Z\big[\frac{1}{p}\big]\big)\) is a lattice in \(\SO_n(\mathbb{R})\times \SO_n(\mathbb{Q}_p)\) in Section~4.

\section{A descriptive view of the \(n\)-spheres}

In this section we introduce the measure theoretic machinery that is fundamental to the main results of this paper.  Then we shall discuss the proof of the following proposition as a warm up.

\begin{proposition}
\label{prop : nonhyperfinite}
The orbit equivalence relation \(\mathcal{R}_3=\mathcal{R}(\SO_{3}(\Q)\curvearrowright\S^{2})\) is not hyperfinite.
\end{proposition}

Clearly, in order to apply Proposition~\ref{prop : amenable} we need to find an invariant probability measure on the sphere.

\subsection{Invariant ergodic measures on the \(n\)-spheres}
If \(G\) is a lcsc group, then there exists a (nonzero) Borel measure \(m_G\)
on \(G\) which is invariant under the action of \(G\) on itself via left translations. Such a measure is called a \emph{Haar measure} on \(G\). If \(m_G'\) is a second Haar measure on \(G\), then there
exists a positive real constant \(c\) such that \(m_G'(A)=cm_G(A)\) for every Borel subset \(A\subseteq G\). It is well-known that \(G\) is compact if and only if \(m_G(G)<\infty\). In this case, there exists a unique Haar probability measure on \(G\). The group \(G\) is said to be \emph{unimodular} if and only if each Haar measure on \(G\) is also invariant under the action of \(G\) on itself via right translations. It is well-known that if \(G\) is either discrete, compact or a connected simple Lie group, then \(G\) is unimodular.

Suppose that \(K\) is a compact second countable group and that \(L< K\) is a closed subgroup. Then both \(K\) and \(L\) are unimodular. Hence there exists a unique \(K\)-invariant probability measure on the standard Borel \(K\)-space \(K/L\). (For example, see \cite[Theorem~3.17]{PlaRap}.) The measure is called the\emph{ Haar probability measure} on \(K/L\). The Haar probability measure on \(K/L\) can be described explicitly as the pushforward of the Haar probability measure \(m_K\) on \(K\) through the canonical surjection \(\pi \colon K\to K/L\).

\begin{lemma}[{See~\cite[Lemma~2.2]{Tho03APAL}}]
\label{lem : K/L} 
Let \(K\) be a compact second countable group, let \(L\leq K\) be a closed subgroup and let \(\mu\) be the Haar probability measure on \(K/L\). If \(\Gamma\) is a countable dense subgroup of \(K\), then 
the action of \(\Gamma\) on \(K/L\) is uniquely ergodic; i.e., \(\mu\) is the unique \(\Gamma\)-invariant
probability measure on \(K/L\).
\end{lemma}

Note that if \(\mu\) is the unique \(\Gamma\)-invariant probability measure and therefore the set of invariant measures consists of only one point,  then \(\mu\) is necessarily ergodic.

Lemma \ref{lem : K/L} can be used to define the invariant measure on \(\mathbb{S}^{n}\). It is well-known that the compact group \(\SO_{n}(\R)\) 
acts transitively on \(\S^{n-1}\).
Then the stabilizer of each point of \(\mathbb{S}^{n-1}\) in \(\SO_n(\R)\) is easily seen to be isomorphic to \(\SO_{n-1}(\R)\).  (In fact, since \(\SO_n(\R)\) acts transitively,  the stabilizers of all points of \(\mathbb{S}^{n-1}\) are conjugate subgroups of \(\SO_n(\R)\).)
Hence, given any fixed \(x_0\in \mathbb{S}^{n-1}\), we can identify \(\S^{n-1}\)  with the coset space \(\SO_{n}(\R)/\SO_{n-1}(\R)\)
via the map \(x\mapsto \{g\in \SO_n(\R) : g\cdot x_0 =x\}\).
Slightly abusing notation, we shall denote both the Haar probability measure on \(\SO_{n}(\R)/\SO_{n-1}(\R)\)
and the corresponding probability measure on \(\S^{n-1}\) by \(\mu_{n}\) and call \(\mu_n\) the \emph{spherical measure}.
Since \(\SO_{n}(\Q)\) is dense in \(\SO_{n}(\R)\) for all \(n\in \N\), it follows that \(\mu_{n}\) is the unique \(\SO_{n}(\Q)\)-invariant measure on \(\S^{n}\) by Lemma~\ref{lem : K/L} .  Thus, the action \(\SO_n(\mathbb{Q})\curvearrowright (\mathbb{S}^{n-1},\mu_n)\) is ergodic.

Notice that it is possible to define \(\mu_n\) in terms of Lebesgue measure.  In fact, if \(\lambda_n\) is the Lebegue measure in \(\mathbb{R}^n\)and \(B_1(0)\) is the unitary \(n\)-dimensional ball, then
\[
\mu_n(A) = \frac{1}{\lambda_n(B_1(0))} \lambda_n(\{ t x : x\in A, 0\leq t \leq 1 \}).
\]
For more details about this approach we refer the  interested reader to \cite[Theorem~3.4--3.7]{Mat}. Now we have all ingredients to prove that \(\mathcal{R}_3\) is not hyperfinite. 

\begin{proposition}
For all \(n\geq 2\)
the free part of the action of \(\SO_{n}(\Q)\) on \(\S^{n-1}\) is conull. That is,
\(\mu_n\big(\Fr(\mathbb{S}^{n-1}) \big) = 1\).
\end{proposition}
\begin{proof}
Consider the action of \(\SO_{n}(\Q)\) on the \(n\)-dimensional euclidean space \(\mathbb{R}^n\).  Let \(\mu=\mu_n\)  be the spherical measure and denote by \(\lambda\) the Lebesgue measure in \(\mathbb{R}^n\). For any nonidentity \(A\in \Gamma\),  the linear subspace of fixed vectors \(\{x\in \mathbb{R}^n\mid Ax = x\}\) is an eigenspace with eigenvalue \(1\), thus has dimension less than \(n\).
It follows that,
 \(\lambda(\{x\in \mathbb{R}^n\mid Ax = x\})=0\). 
If \(A\in \Gamma\) is not the identity, denote by \(\Fix_A=\{x\in \mathbb{S}^{n-1}: Ax=x\}\).
 Clearly \[\mu(\Fix_A) = \frac{1}{\lambda(B_1(0))} \lambda\big(\big\{tx \mid x\in \mathbb{S}^{n-1}, \, Ax=x \text{ and } t\in [0,1]\big\}\big) =0.\] 
Therefore,
\[
\mu(\Fr(\mathbb{S}^{n-1})) = \mu\bigg( \mathbb{S}^{n-1} \setminus \bigcup_{A\in \Gamma\setminus \{1\}} \Fix_A\bigg) = 1.\qedhere
\]
\end{proof}

\begin{proof}[Proof of Proposition~\ref{prop : nonhyperfinite}]
Let \(\mu = \mu_{3}\) be the spherical measure on \(\mathbb{S}^2\), and
consider the orbit relation \(\mathcal{R}_3\) induced by the action of \(\SO_{3}(\Q) \) on \(\S^{2}\).  Because of the previous discussion \(\SO_3(\Q)\) acts on \((\mathbb{S}^{2},\mu)\) ergodically.
It is well-known that \(\SO_{3}(\Q)\) contains a free subgroup \(\Gamma\) of rank \(2\).
E.g., as pointed out by Tao~\cite{Tao04} we can define \(\Gamma = \langle\sigma,\tau\rangle\) the group generated by the matrices
\[
\sigma = \frac{1}{5}
\begin{bmatrix} 
3 & 4 & 0 \\
-4 & 3 & 0\\
0 & 0 & 5 \\
\end{bmatrix}
\quad
\text{and}\quad
\tau = \frac{1}{5}
\begin{bmatrix} 
5 & 0 & 0 \\
0 & 3 & 4\\
0 & -4 & 3 \\
\end{bmatrix}.
\]
Now, by Proposition~\ref{prop : containment}  it suffices to show that the orbit equivalence relation induced by the action of \(\Gamma\) is not hyperfinite.
Let \(Z\) be the set of points in \(\mathbb{S}^{2}\) that are fixed by some non-identity transformation in \(\Gamma\), that is \( Z = \setm{x\in \S^{2}}{\exists \gamma\in \Gamma\smallsetminus\{1\}\; \gamma\cdot x = x}\).
Each non-identity trasformation in \(\Gamma\) is a rotation with exactly two fixed points on \(\S^{2}\), namely, the ones along its rotation axis. Consequently,  \(Z\) and its saturation \(\Gamma\cdot Z\) are countable sets,  thus \(\mu(\Gamma\cdot Z)=0\).  Let \(M\coloneqq \mathbb{S}^2\setminus \Gamma\cdot Z\). It is clear that \(M\) is \(\Gamma\)-invariant.
Since \(\Gamma\) is not amenable, the induced equivalence relation \(\mathcal{R}(\Gamma\curvearrowright M)\) is not hyperfinite by Proposition~\ref{prop : amenable}.
It follows that \(\mathcal{R}_3\) is not hyperfinite by Proposition~\ref{prop : containment}~\ref{item : b} as desired.
\end{proof}

\begin{remark}
In fact,  our argument shows that \(\mathcal{R}\big( \SO_3\big(\Z\big[\frac{1}{5}\big]\big) \curvearrowright \mathbb{S}^{2} \big)\) is not hyperfinite.
\end{remark}

Notice that \(\SO_2(\Q)\) is the group of rotations of the plane about the origin with rational entries. Hence \(\SO_2(\Q)\) is
isomorphic to a subgroup of the abelian group \(\R/\Z\). It readily follows that the orbit equivalence relation induced by \(\SO_2(\Q)\curvearrowright \mathbb{S}^1\) is hyperfinite because of Theorem~\ref{thm : GaoJac}.

The fact that \(\mathcal{R}_2\) is not concretely classifiable is a consequence of the following classical result in descriptive set theory.

\begin{proposition}[{E.g., see ~\cite[Corollary~3.5]{Hjo00}}]
\label{Prop : dense orbit}
Let \(G\) be a countable group acting by homeomorphisms on a Polish space \(X\).  If there is a dense orbit and every orbit is meager, then the induced orbit equivalence relation \(\mathcal{R}(G\curvearrowright X)\) is not smooth.
\end{proposition}

\begin{proposition}
\(\mathcal{R}_2\) is not concretely classifiable. Thus, \(\mathcal{R}_2 \sim_B E_0\).
\end{proposition}

\begin{proof}
Let \(A=\frac{1}{5}
\begin{bsmallmatrix} 4 & -3\\ 3 & 4 \end{bsmallmatrix}\). It is easily checked that \(A\in \SO_2(\Q)\), and \(A\) is the rotation of \(\mathbb{S}^1\) by the angle \(\theta\) with \(\sin \theta = \frac{3}{5}\) and \(\cos \theta = \frac{4}{5}\). 
That is, \(A\cdot x = xe^{2\pi i r}\) for all \(x\in \mathbb{S}^1\) and for \(r\in[0,1)\) with \(\theta = 2\pi r\).
By Niven's theorem ~\cite[Corollary~3.12]{Niv56}, it follows that \(r\) is an irrational number.
So, let \(\Gamma=\langle A\rangle\) be the subgroup of \(\SO_2(\mathbb{Q})\) generated by \(A\). 
Since \(r\) is irrational,  we know that for any point \(x\in\mathbb{S}^{1}\) the orbit \(\Gamma\cdot x\) is dense in \(\mathbb{S}^1\).
It follows that \(\mathcal{R}(\Gamma\curvearrowright \mathbb{S}^1)\) is not concretely classifiable by Proposition~\ref{Prop : dense orbit}. 
\end{proof}

The results  of next section generalize and strengthen Proposition~\ref{prop : nonhyperfinite}. In fact,  we shall prove that for all \(n\geq3\) the orbit equivalence relation of \(\mathcal{R}_n\) is not treeable. 

\subsection{Non-treeability through property \((\mathrm{T})\)}


We recall the definition of property \((\mathrm{T})\) introduced by Kazhdan~\cite{Kaz67}. Let \(G\) be a lcsc group and let \(\pi\colon G \to U(\mathcal{H})\) be a unitary representation of \(G\) on the separable Hilbert space \(\mathcal{H}\). We say that \(\pi\) \emph{almost admits invariant vectors} if for every \(\epsilon > 0\) and every compact subset \(K \subseteq G\),  there exists a unit
vector \(v \in \mathcal{H}\) such that \(\lVert \pi(g)\cdot v - v \rVert < \epsilon\) for all \(g \in K\).
We say that \(G\) has \emph{Kazhdan property \((\mathrm{T})\)} if for every unitary representation \(\pi\) of \(G\),  if \(\pi\) almost admits invariant vectors, then \(\pi\) has a non-zero invariant vector. It is well-known that if \(\Gamma\) is a countable group and \(\Gamma\) has property \((\mathrm{T})\), then \(\Gamma\) is finitely generated.

The first ingredient in the proof of Theorem~\ref{thm : not treeable} is the following antitreeability result that was isolated by Hjorth and Kechris in~\cite[Theorem~10.5]{HjoKec}.

\begin{theorem}[essentially Adams-Spatzier~\cite{AdaSpa}]
\label{thm : AdaSpa}
Let \(\Gamma\) be a countable  group with Kazhdan's property \((\mathrm{T})\), \(X\) a standard Borel \(\Gamma\)-space and \(\mu\) a \(\Gamma\)-invariant, nonatomic, ergodic measure on \(X\). Then
\(\mathcal{R}(\Gamma\curvearrowright X)\) is not treeable.
 (In fact, if \(F\) is any treeable countable Borel equivalence relation, then \(\mathcal{R}(\Gamma\curvearrowright X)\) is \(F\)-ergodic.)
\end{theorem}

\begin{remark}
Hjorth~\cite{Hjo99} pointed out that the hypothesis that the action of \(\Gamma\curvearrowright (X,\mu)\) be ergodic is not necessary. The weaker assumption that \(\mathcal{R}(\Gamma\curvearrowright X)\) is not smooth suffices.
\end{remark}

Also, we shall use the next result, which is contained also in~{\cite[Theorem~6.4.4]{BekdeLVal}.

\begin{theorem}[Margulis~{\cite{Mar80}}]
\label{thm : Margulis dense Kazhdan}
Let \(p\) be a prime
number with \(p \equiv 1 \pmod 4\). For \(n > 4\), the group
\(\SO_{n}{\big(\Z\big[\frac{1}{p}\big]\big)}\) is a dense 
 subgroup of \(\SO_{n}(\R)\) with Kazhdan's property~\((\mathrm{T})\).
\end{theorem}

\begin{proposition}
\label{thm : antitreable >4}
For \(n\geq 5\), \(\mathcal{R}_n\) is not treeable.
\end{proposition}
\begin{proof}
Fix \(n\geq 5\) and let \(\Gamma = \SO_{n}{\big(\Z\big[\frac{1}{5}\big]\big)}\).  Since the class of treeable equivalence relations is closed under containement it suffices to show that the orbit equivalence relation \(\mathcal{R}(\Gamma\curvearrowright X)\) is not treeable.  We emphasize that  \(\Gamma\) is a dense subgroup of \(\SO_n(\Q)\) by Theorem~\ref{thm : Margulis dense Kazhdan},  therefore the action of \(\Gamma\curvearrowright (\mathbb{S}^{n-1}, \mu_n)\) is ergodic because of Lemma~\ref{lem : K/L}.  Then, since \(\Gamma\) has property \((\mathrm{T})\), it follows \(\mathcal{R}(\Gamma\curvearrowright X)\) is not treeable by Theorem~\ref{thm : AdaSpa}. 
\end{proof}

The proof of Proposition~\ref{thm : antitreable >4} does not generalize to all \(n\geq 3\). In fact,  Zimmer~\cite{Zim84} proved that \(\SO_{n}(\R)\) contains no infinite countable Kazhdan subgroup for \(n = 3, 4\).  
Notice that Kechris~{\cite[Theorem~9]{Kec00}} proved that every non-amenable lattice \(\Gamma\)
in a product group \(G = G_{1} \times G_{2}\) of two locally compact second countable groups, each of which contains an infinite amenable discrete group,  is \emph{anti-treeable}. That is,  for every Borel action of \(\Gamma\) on a standard Borel space \(X\),  which is free and admits an invariant probability Borel measure,  the induced equivalence relation \(\mathcal{R}(\Gamma\curvearrowright X)\) is not treeable.
Kechris' antitreeability result does not rely on property \(\mathrm{(T)}\). However, it does not apply to our situation. While,  \(\SO_n\big(\Z\big[\frac{1}{5}\big]\big)\) is a lattice in \(\SO_n(\mathbb{R})\times \SO_n(\mathbb{Q}_5)\), every discrete subgroup of \(\SO_n(\mathbb{R})\) is finite because \(\SO_n(\mathbb{R})\) is compact. We shall derive the non-treeability of \(\mathcal{R}_3\) and \(\mathcal{R}_4\) from a cocycle superrigidity theorem in the next section. (See Corollary~\ref{cor : antitreeable 3,4} below.)

\section{Cocycle superrigidity}
Before proving the main results of this paper we briefly introduce the notion of cocycle.

\subsection{Cocycles}
Let \(\Gamma\) be a countable group and let \(X\) be a standard Borel \(\Gamma\)-space with an invariant probability measure \(\mu\).
\begin{definition}
Suppose that \(\Delta\) is a countable group.
Then a Borel function \(\alpha\colon \Gamma \times X \to \Delta\) is called a \emph{cocycle} if for all \(g,h\in \Gamma\) 
\[
\alpha(hg,x) = \alpha(h,g\cdot x)\alpha(g,x)\qquad \mu\text{-a.e.}(x).
\]
Moreover, the cocycle is called \emph{strict} if this equation holds for all \(g,h \in \Gamma\) and \(x\in X\).
\end{definition}

Suppose that \(X_0\subseteq X\) is a \(\Gamma\)-invariant set and \(\alpha\colon \Gamma\times X_0\to \Delta\) is a strict cocycle. Then we can obviously extend \(\alpha\) to a strict cocycle \( \bar\alpha\colon \Gamma\times X\to \Delta\) by setting
\(\bar \alpha(g,x) = \alpha(g,x)\) for all \(x\in X_0\),  and
\(\bar \alpha(g,x)=1_\Delta\) for all \(x\notin X_0\).  Further,  since \(\Gamma\) is assumed to be countable,  
for every cocycle \(\alpha\colon \Gamma\times X \to \Delta\), there is a strict cocycle \(\alpha'\) such
   that for all \(g\in \Gamma\)
   \[
   \alpha(g, x) = \alpha'(g, x)\qquad \mu\text{-a.e.}(x).
   \]
    In fact, if we let \(A_{g,h} \coloneqq \{x \mid \alpha(hg,x) = \alpha(h,g\cdot x)\alpha(g,x)\}\), for all \(g,h\in \Gamma\),  and \(A\coloneqq \bigcap_{g,h\in \Gamma} A_{g,h}\),  then \(X_0 = \bigcup_{g\in \Gamma} g\cdot A\) is a \(\Gamma\)-invariant Borel set with \(\mu(X_0) = 1\) and \(\alpha\restriction \Gamma\times X_0\) is strict.
Then we can define \(\alpha'\) by extending \(\alpha\restriction \Gamma\times X_0\) as described above. Therefore, we shall assume that every cocycle is strict.

Cocycles typically arise in the following manner. Let \(\Gamma\) and \(\Delta\) be countable discrete groups.  Let \(X\), \(Y\) be a standard Borel \(\Gamma\)-space and a \(\Delta\)-space, respectively.
Also, suppose that  the action \(\Delta\curvearrowright Y\) is free. 
If \(f\colon X \to Y\) is a Borel homomorphism from \(\mathcal{R}(\Gamma\curvearrowright X)\) to \(\mathcal{R}(\Delta\curvearrowright Y)\),  then we can define a Borel cocycle \(\alpha\colon \Gamma \times X \to \Delta \) by letting \(\alpha(g, x)\) be the unique element of \(\Delta\) such that
\[\alpha(g,x)\cdot f(x) =f(g\cdot x).\]
In this case we say that \(\alpha\) is \emph{the cocycle associated to \(f\)}.
Further,  notice that whenever \(\alpha(g, x) = \alpha(g)\) only depends on \(g\), then \(\alpha \colon \Gamma \to \Delta\) is a group homomorphism.
Suppose now that \(B\colon X \to \Delta\) is a Borel function and that \(f'\colon X \to Y\) is defined by \(f'(x) = B(x)\cdot f(x)\). Then \(f'\) is a Borel homomorphism from \(\mathcal{R}(\Gamma\curvearrowright X)\) to \(\mathcal{R}(\Delta\curvearrowright Y)\) as well, and the corresponding cocycle
\(\beta \colon \Gamma\times X \to \Delta\) satisfies
\[
\beta(g,x) = B(g\cdot x)\alpha(g,x)B(x)^{-1}
\]
for all \(g\in \Gamma \) and \(x \in X\). 
The above equation states that \(\alpha\) and \(\beta\) are equivalent as cocycles (or cohomologous). This notion is made precise by the following definition.

\begin{definition}
\label{def : cohomologous}
Suppose that \(\Delta\) is a countable group. 
Then the cocycles \(\alpha,\beta\colon \Gamma \times X \to \Delta\) are \emph{cohomologous} (or \emph{equivalent}) iff there exist a Borel function \(B\colon X \to H\) and a \(\Gamma\)-invariant Borel subset
\(X_{0}\subseteq X\) with \(\mu(X_{0})=1\) such that
\[
\beta(g,x) = B(g\cdot x)\alpha(g, x)B(x)^{-1}
\]
for all \(g\in \Gamma\) and \(x\in X_{0}\).
\end{definition}

Analyzing cocycles that arise from Borel homomorphisms could be useful to prove an irreducibility result, namely, to prove that a given countable Borel equivalence relation is not Borel reducible to another one.
With the same notation as above, suppose that \(f\colon X \to Y\) is a Borel homomorphism from \(\mathcal{R}(\Gamma\curvearrowright X)\) to \(\mathcal{R}(\Delta\curvearrowright Y)\) and  \(\alpha\) is the cocycle associated to \(f\).
If \(\alpha\) is cohomologous to \(\beta\) and there are further restrictions on \(\beta\), we might be able to use ergodicity to find serious obstructions for a Borel reduction from
\(\mathcal{R}(\Gamma\curvearrowright X)\) to \(\mathcal{R}(\Delta\curvearrowright Y)\), as shown in the following lemma.

\begin{lemma}
\label{lem : fin_subg}
Let \(\Gamma\) and \(\Delta\) be countable groups and let \(\Gamma,\Delta\) act in a Borel fashion on the standard Borel spaces \(X,Y\), respectively. Let \(\mu\) be a \(\Gamma\)-invariant probability measure on \(X\).
Suppose that \(f\colon X\to Y\) is a Borel homomorphism from \(\mathcal{R}(\Gamma\curvearrowright X)\) to \(\mathcal{R}(\Delta\curvearrowright Y)\), and let \(\alpha\colon \Gamma\times X \to \Delta\) be the cocycle associated to \(f\).
If \(\alpha\) is cohomologous to the cocycle \(\beta\colon \Gamma\times X \to \Delta\) and  \(\beta(\Gamma\times X)\) is finite, then there exists a Borel \(M\subseteq X\) with \(\mu(M)=1\) such that \(f(M)\) is contained in a single \(\mathcal{R}(\Delta\curvearrowright Y)\)-class.
\end{lemma}

\begin{proof}
Suppose \(\alpha\) and \(\beta\) are cohomologous,  and \(\beta(\Gamma\times X)\) is contained in a finite \(K\subseteq \Delta\). Let \(B\colon X\to \Delta\) and \(X_{0}\subseteq X\) such that
\[
\beta(g,x) = B(g\cdot x) \alpha(g,x) B(x)^{-1}
\]
for all \(g\in \Gamma\) and \(x\in X_{0}\).
Then define \(f'\colon X\to Y\) by setting \(f'(x)= B(x)\cdot f(x)\).
For all \(x\in X_{0}\), let
\begin{align*}
\Upphi(x) =& \{ \beta(g,x)\cdot f'(x) \mid g\in \Gamma \}\\
	   =& \{ f^{'}(z) \mid (z,x)\in\mathcal{R}(\Gamma\curvearrowright X) \}.
\end{align*}
is a nonempty finite subset of \(Y\) and if \((x,y)\in\mathcal{R}(\Gamma\curvearrowright X)\) then \(\Upphi(x)=\Upphi(y)\).
Therefore, \(\Upphi\) is a \(\Gamma\)-invariant Borel map into \(Y^{<\omega}\). Since \((X,\mu)\) is ergodic,  then there is a \(\Gamma\)-invariant \(M\subseteq X_{0}\) with \(\mu(M)=1\) such that \(\Upphi\restriction M\) is constant.  It follows that \(f\) maps \(M\) into a single \(\mathcal{R}(\Delta\curvearrowright Y)\)-class.  
\end{proof}

\subsection{Cocycle superrigidity}
Now we discuss some cocycle supperrigidity results. Let \(G\) be a connected compact group with Haar measure \(m_G\).
Suppose that \(\Gamma\) is a countable group together with a dense embedding \(\Gamma\hookrightarrow G\) so that the left-translation action \(\Gamma\curvearrowright (G,m_G)\) is ergodic.  Let \(\alpha\colon \Gamma\times G\to \Delta\) be a cocycle.

If \(\Gamma\) has Kazhdan  property \((\mathrm{T})\), then  there exist a neighborhood \(V\) of \(1_G\) and a constant \(C \in (31/32, 1)\) such that \(m_G(\{x \in G \mid \alpha(g, xt) = \alpha(g, x)\}) \geq C\) for all \(g \in \Gamma\) and every \(t\) in \(V\).  (See Ioana~\cite[p.~2742]{Ioa16}.) This ``local uniformity'' condition has many useful consequences. In particular, we are interested in the following:

\begin{theorem}
\label{thm : superrigidity}
Let \(\Gamma\leq G\) be a dense subgroup of a connected, compact group \(G\). Suppose that \(\Gamma\) has Kazhdan property $\mathrm{(T)}$. Consider the left translation action \(\Gamma \curvearrowright(G, m_{G})\), where \(m_{G}\) is the Haar measure of \(G\). Assume that \(\pi_{1}(G)\), the fundamental group of \(G\), is finite. Let \(\Lambda\) be a countable group and \(\alpha \colon \Gamma \times G \to \Lambda \) be a cocycle.
If any group homomorphism \(\pi_{1}(G) \to \Lambda\) is trivial, then \(\alpha\) is cohomologous to a homomorphism \(\delta \colon \Gamma \to \Lambda\).
\end{theorem}

Theorem~\ref{thm : superrigidity} is a special case of a result of Ioana~\cite[Theorem~3.2]{Ioa16}, that was abstracted from an argument of Furman~\cite[Theorem~5.21]{Fur11}.

Drimbe and Vaes recently proved cocycle superrigidity for translation actions \(\Gamma \curvearrowright  G\)
without relying on property $\mathrm{(T)}$.  We introduce some terminology of their paper and discuss some applications that are relevant to our results.

\begin{definition}
Let \(\Gamma\) be a dense subgroup of a lcsc group \(G\).  We say that the translation action \(\Gamma\curvearrowright G\) is \emph{cocycle superrigidity (with countable targets)} if every cocycle
\(\alpha \colon \Gamma \times G \to \Lambda\) for the translation action into any countable group \(\Lambda\) is cohomologous to a group homomorphism.
\end{definition}

\begin{definition}
Let  \(\Gamma \) and \(G\) be as above.  We say that the translation action \(\Gamma\curvearrowright G\) is  \emph{essentially cocycle
superrigid (with countable targets)} if for any cocycle \(\alpha \colon \Gamma \times G \to \Lambda\) for the
translation action with values in an arbitrary countable group \(\Lambda\), there exists an open subgroup
\(G_0 < G\) and a covering \(\pi\colon \widetilde{G} \to G_0\) such that if \(\widetilde \Gamma = \pi^{-1}(\Gamma \cap G_0)\),  then the lifted cocycle
\(\widetilde \alpha \colon\widetilde{\Gamma} \times \widetilde G \to \Lambda,  \widetilde \alpha(g,  x) = \alpha(\pi(g), \pi(x))\)
is cohomologous to a group homomorphism \(\widetilde \Gamma \to \Lambda\).
\end{definition}

\begin{proposition}[{See~\cite[Proposition~3.1]{DriVae}}]
 If \(G\) is a connected Lie group with universal cover \(\pi \colon \widetilde G\to G\), then \(\Gamma < G\) is essentially
cocycle superrigid with countable targets if and only if for every cocycle \(\alpha \colon \Gamma \times G \to \Lambda\)
with values in a countable group \(\Lambda\),  the lifted cocycle \(\widetilde \alpha\colon \pi^{-1}
(\Gamma) \times \widetilde G \to \Lambda, \widetilde \alpha(g,  x) = \alpha(\pi(g), \pi(x))\)  is cohomologous to a group homomorphism \(\pi^{-1}( \Gamma) \to \Lambda\).
\end{proposition}

The following is a special case of \cite[Proposition~4.1]{DriVae}.

\begin{theorem}
\label{thm : superrigidityDV}
Let \(S=\{p,q\}\)  for some distinct odd prime numbers \(p\) and \(q\).  Then the translation action \(\SO_n(\mathbb{Z}[1/S])\curvearrowright \SO_n(\mathbb{R})\) is essentially cocycle superrigid with countable target.
\end{theorem}

The next antitreeability result for homogeneous spaces is essentially contained in \cite[Scections~2--3]{DriVae}.  
Since this formulation might be a useful reference in the future, 
we rework the statement and sketch the proof for completeness 

\begin{proposition}[\cite{DriVae}]
\label{prop : antitreeabilityDV}
Let \(G\) be a connected Lie group with dense subgroup \(\Gamma < G\) and universal
cover \(\pi \colon \widetilde G \to G\).  Let \(P < G\) be a closed subgroup with \(\pi^{-1}(P)\) connected.  If \(\Gamma<G\) is essentially cocycle supperrigid, then \(\mathcal{R}(\Gamma\curvearrowright G/P)\) is not treeable.
\end{proposition}
\begin{proof}
If \(\Gamma<G\) is essentially cocycle supperrigid, then the action \(\Gamma\curvearrowright G/P\) is cocycle superrigid with countable targets (\cite[Proposition~3.3]{DriVae}).   Then the action \(\Gamma\curvearrowright G/P\) is OE-superrigid in the sense of \cite[Definition~2.1]{DriVae}.   Then we can prove that \(\mathcal{R}(\Gamma\curvearrowright G/P)\) is not treeable  as in the second paragraph of \cite[Section~2.3]{DriVae}.  By a result of Hjorth~\cite[Corollary~1.2]{Hjo06}, after possibly deleting a null set, there is some \(n\in\{n\in\mathbb{N}\mid n\geq 2\}\cup\{\infty\}\) and a free p.m.p. action \(\mathbb{F}_n\curvearrowright (G/P,\mu) \) such that \[\mathcal{R}(\Gamma\curvearrowright G/P)=\mathcal{R}(\mathbb{F}_n\curvearrowright G/P)).\] However,  by a theorem of Monod and Shalom~\cite[Theorem~2.27]{MonSha} the action \(\mathbb{F}_n\curvearrowright (G/P,\mu)\) is orbit equivalent to actions of uncountably many nonisomorphic
groups.
\end{proof}

\begin{corollary}
\label{cor : antitreeable 3,4}
\(\mathcal{R}_3\) and \(\mathcal{R}_4\) are not treeable.
\end{corollary}

\begin{proof}
For \(n =3,4\), let \(\Gamma = \SO_n(\mathbb{Z}[1/S])\) where \(S\) is any set of two distinct odd prime numbers.  It follows from Theorem~\ref{thm : superrigidityDV} and Proposition~\ref{prop : antitreeabilityDV} that \(\mathcal{R}(\Gamma\curvearrowright \mathbb{S}^{n-1})\) is not treeable. Since \(\mathcal{R}(\Gamma \curvearrowright \mathbb{S}^{n-1})\subseteq \mathcal{R}_n\),  we conclude that \(\mathcal{R}_n\) is not treeable by Proposition~\ref{prop : containment}~\ref{item : a}.
\end{proof}

\begin{remark}
As pointed out by the anonymous referee the proof of Corollary~\ref{cor : antitreeable 3,4} works for all \(n\geq 3\).
\end{remark}

\subsection{\(S\)-arithmetic groups}

Before introducing \(S\)-arithmetic groups and discussing their properties we recall some basic facts and terminology about valuations. 

A \emph{global field} is a finite extension of either the field of rational numbers or of a field of rational functions in one variable over a finite field. Throughout this section let  \(k\) be a global field.  A \emph{valuation} (or \emph{absolute value}) on \(k\) is a function \(v\colon k\to \mathbb{R}^+\cup\{0\}\) such that
\begin{itemizenew}
\item
\(v(x)=0\) if and only if \(x=0\);
\item
\(v(xy)=v(x)v(y)\) for all \(x,y\in k\);
\item
\(v(x+y)\leq v(x)+v(y)\) for all \(x,y\in k\).
\end{itemizenew}
To avoid trivial valuation we also assume that \(v(x)\neq 1\) for some \(x\in k\), \(x\neq 0\).
Notice that any valuation \(v\) induces a metric \(d_v\) on \(k\) by setting \(d_v(x,y) = v(x-y)\) for all \(x,y\in k\).  We denote by \(k_v\) the completion of \(k\) relative to this metric.  The field \(k_v\) is called a \emph{local field}. Two valuations \(v\) and \(w\) on the same field are said to be \emph{equivalent} precisely when the corresponding metric \(d_v\) and \(d_w\) induce the same topologies.

A valuation \(v\) is said to be \emph{non-archimedean}  if and only if it satisfies the ultrametric inequality. Otherwise, we say that \(v\) is \emph{archimedean}. 
Note that a global field \(k\) admits an archimedean valuation if and only \(\operatorname{char} k=0\).
We shall only consider global field \(k\) with \(\operatorname{char} k =0\),  thus the set of archimedean valuation is non-empty.
A valuation \(v\) is archimedean if and only if \(k_v\) is archimedean.  (A local field \(k\) is said to be \emph{archimedean} if and only if it is isomorphic to \(\mathbb{R}\) or \(\mathbb{C}\).)
Moreover,  notice that \(k_v\) is not archimedean if and only if it is totally disconnected. 

The set \(\{x\in k\mid v(x)\leq 1\}\) is called the \emph{ring of integers} \(k\). This is the unique maximal compact subring \(k\). When \(v\) is non-archimedean the ring of integers of \(k_v\) is usually denoted by \(\mathcal{O}_v\). More generally, suppose that \(S\) is a finite set of valuation of \(k\). An element \(x\in k\) is said to be \emph{\(S\)-integral} if and only if \(v(x) \leq 1\) for each non-archimedean valuation \(v\notin S\).  The set of \(S\)-integral elements is a subring of \(k\) and is denoted by \(k(S)\).

In this paper we mainly deal with the case
 \(k= \mathbb{Q}\). 
To discuss valuations uniformly, it
we let \(\mathbb{Q}_\infty = \R\).
 
 For \(p=\infty\), let \(v_p(x)=|x|\) be the standard absolute value.  If \(p\) is a prime number, then \(v_p\) is the standard \(p\)-adic valuation. That is,  we set \(v_p(x)=p^{-\ell}\), where \(x=p^\ell \frac{a}{b}\), and \(a,b\) are coprime with \(p\).  (We let \(v_p(0)=0\)). 
 It is well-known that every valuation on \(\mathbb{Q}\) is equivalent to \(v_p\) for some \(p\in \{p\mid \text{\(p\) is prime}\}\cup\{\infty\}\).  Clearly, \(v_p\) is archimedean if and only if \(p=\infty\).
An element of \(x\in \mathbb{Q}\) is \(S\)-integral if and only if \(x=\frac{m}{n}\) where \(m\in \mathbb{Z}\) and \(n\in \mathbb{N}\)  and the prime factors of \(n\) are in \(S\).
For instance, if \(S=\{p,\infty\}\), then \(\mathbb{Q}(S)= \mathbb{Z}\big[\frac{1}{p}\big]\).

If \(G\) is a connected algebraic \(\mathbb{Q}\)-group,  and \(S\) is a finite subsete of the set of valuations \(\{p\mid p\text{ prime}\}\cup \{\infty\}\) with \(\infty \in S\),  then \(G(\mathbb{Q}(S))\) embeds into \(\prod_{p\in S} G(\mathbb{Q}_p)\) as a discrete subgroup.  In fact, \(G(\mathbb{Q}(S))\) is a lattice in \(\prod_{p\in S}G(\mathbb{Q}_p)\).  More generally, we consider the following definition.

\begin{definition}
Let \(k\) be a global field and \(S\) be a set of inequivalent valuations on \(k\).  Also let \(G\) be a \(k\)-group.
A subgroup of \(G\) is said to be \emph{\(S\)-arithmetic} if and only if it is commensurable with \(G(k(S))\).
\end{definition}

Suppose that \(G,H\) are \(k\)-groups. A  \(k\)-homomorphism \(f\colon G \to H\) is said to be a \emph{\(k\)-isogeny} if and only if it is surjective and has finite kernel.
We recall that for every connected semisimple \(G\)
there exist algebraic groups \(\widetilde G\),  and an isogeny \(\pi\colon \widetilde G \to G\) so that
for every isogeny \(\rho\colon H \to G\),  there is an isogeny \(\pi_\rho \colon \widetilde G\to H\) such that the composition \(\rho\circ \pi_\rho = \pi\). (E.g., see \cite[Proposition~1.4.11]{Mar})
The group \(\widetilde G\) is called the \emph{(algebraic) universal covering} of \(G\).  If G is a \(k\)-group,  so is \(\widetilde G\), and \(\pi\) is defined over \(k\).

Next proposition is a special case of Margulis~\cite[Theorem~3.2.9, Chapter~1]{Mar}.

\begin{proposition}
\label{prop : covering}
Let \(S\) be a finite set of valutations of the field \(k\).  Suppose \(G\) is a connected simple \(k\)-group. If \(\Lambda\) is an \(S\)-arithmetic subgroup of \(G\),  then \(\pi^{-1}(\Lambda)\) is an \(S\)-arithmetic subgroup of \(\widetilde G\).
\end{proposition}

The following is known as Margulis' normal subgroup theorem for \(S\)-arithmetic groups. We let
\[\rank_S G = \sum_{p\in S\cup\{\infty\}} \rank_{k_p} G.\]

\begin{theorem}[Margulis, {\cite[Theorem~A, Chapter~VIII]{Mar}}]
\label{thm : nst}
Suppose that \(G\) is a connected almost \(k\)-simple \(k\)-group.
Let \(\Gamma\) be an \(S\)-arithmetic subgroup of \(G\) and let \(N\) be a normal subgroup of \(\Gamma\). Suppose \(\rank_S G \geq 2\) and that either \(G\) is connected or \(S\) is finite. Then either \(N\) 
lies in the centre of \(G\) or the quotient \(\Gamma/N\) is finite (in which case \(N\) has finite index in \(\Gamma\)).
\end{theorem}

We state a particular case of the famous Margulis' superrigidity theorem for \(S\)-arithmetic groups (\cite[Theorem~5.14, Chapter VII]{Mar}). 

\begin{theorem}[Margulis]
\label{Margulis : superrigidity}
Suppose that for each \(p \in S\),  \(G_p\) is a simply connected,  semisimple \(\mathbb{Q}_p\)-group such that \(G_p(\mathbb{Q}_p)\) has no anisotropic factors.  Suppose also that \(\Gamma\) is an irreducible lattice in \(G = \prod_{p\in S\cup\{\infty\}}G_p(\mathbb{Q}_p)\) and \(\rank_S G \geq 2\).  Let \(k\) be  \(\mathbb{Q}_q\) for some prime number \(q\) and \(H\) be a connected almost \(k\)-simple \(k\)-group.

Suppose that \(\rho\colon \Gamma \to H(k)\) is a homomorphism with \(\rho(\Gamma)\) Zariski dense in \(H\).

\begin{enumerate-(a)}
\item
\label{superrigidity : (a)}
If \(q\notin S\), then \(\rho(\Gamma)\) is relatively compact in \(H(k)\).
\item
\label{superrigidity : (b)}
If \(q\in S\) and \(\rho(\Gamma)\) is not relatively compact in \(H(k)\), then there is a uniquely determined \(k\)-epimorphism \(\hat \rho \colon  G\to H(k)\) and a uniquely determined homomorphism \(\nu\colon \Gamma \to Z(H)\) such that
\(\rho(g) = \nu(g)\cdot \hat\rho(g)\), for all \(g\in \Gamma\).
\end{enumerate-(a)}
\end{theorem}

\begin{remark}
In the sequel we apply Theorem~\ref{Margulis : superrigidity} for \(G_p = \widetilde {\SO_n} = \Spin_n\).
Recall that, for \(n=3\) or \(n\geq 5\), 
 the algebraic group \(\SO_n\) is almost \(k\)-simple for \(k=\mathbb{Q}_p\), and any prime \(p\).
(E.g., see Margulis~\cite[Chapter~IX, Remark~1.7(vi)]{Mar}.)
It follows that \(\Spin_n\) is almost \(k\)-simple as well.
(If \(G\) is connected and \(\mathbb{Q}_p\)-simple and \(N\subseteq \widetilde G\) is a normal \(\mathbb{Q}_p\)-subgroup then \(\pi(N)\) is a normal \(\mathbb{Q}_p\)-subgroup of \(G\) which is eihter finite or the whole \(G\). Since \(\pi\) has finite kernel,  in the first case \(N\) is finite, while in the second case \(N\) has finite index in \(\widetilde G\).  However, since \(\widetilde G\) is connected, \(\widetilde G=\widetilde G^0 \) which is known to be the minimal closed subgroup of finite index in \(\widetilde G\).)
So,  the hypothesis that \(G_p(\mathbb{Q}_p)\) has no anisotropic factors is always satisfied in our situation.
\end{remark}

Also, in the proof of Theorem~\ref{thm : irreducibility free} we shall use the following well-known result.  (E.g., see~\cite[Corollary~1.3.5]{BekdeLVal}.)

\begin{proposition}
\label{prop : (T) to amenable}
Let \(G\) be a topological group with Kazhdan Property \((\mathrm{T})\),  and let \(H\) be a locally compact amenable group.  If \( \rho \colon G \to H\) is a continuous homomorphism then \(\rho(G)\) is relatively compact.
\end{proposition}

\subsection{Proof of Theorem~\ref{thm : irreducibility free}}
Recall that we identify \(\mathbb{S}^{n-1}\) with the coset space \(\SO_n(\R)/\SO_{n-1}(\R)\). If \(x\in \SO_n(\R)\), we denote by \([x]\) the coset \(x\SO_{n-1}(\R)\) for readability.
Let \(n\geq 5\) and \(m<n\).
Suppose that \(f\colon \mathbb{S}^{n-1} \to \mathbb{S}^{m-1}\) is a Borel map such that \([x]\mathbin{\mathcal{R}_n}[y] \iff f([x])\mathbin{\mathcal{R}^*_m}f([y])\) for all \([x],[y]\in \mathbb{S}^{n-1}\).

Put \(\Gamma = \SO_{n}\big(\Z\big[\frac{1}{5}\big]\big)\),  \(\mu=\mu_n\), and \(\Delta = \SO_m(\mathbb{Q})\).
 Clearly, \(f\) is also a countable-to-one homomorphism from \(\mathcal{R}(\Gamma\curvearrowright \mathbb{S}^{n-1})\) into \(\mathcal{R}^*_m\).
 Since \(\mathcal{R}^*_m\) is free, let \(\omega \colon \Gamma \times \mathbb{S}^{n-1} \to \Delta\) be the Borel cocycle associated to \(f\).
 
 Let \(G \coloneqq \SO_n(\R)\) and denote by \(m_G\) the Haar measure on \(G\).
 Since
 the action of \(\Gamma\) on \(S^{n-1}=\SO_n(\mathbb{R})/\SO_{n-1}(\mathbb{R})\)  is a quotient of the left-translation action of \(\Gamma\) on \(\SO_n(\mathbb{R})\), the cocycle
 \(\omega\) induces a Borel cocycle \(\alpha\colon\Gamma \times G \to \SO_{m}(\Q)\)  for the left-translation action \(\Gamma\curvearrowright (G,m_{G})\), which is defined by setting \(\alpha(g,x)=\omega(g, [x])\).

As usual let \(\widetilde G\) be the universal covering group of \(G\), and let \(p \colon \widetilde G\to  G\) the \(2\)-\(1\) covering map. Also let \(\widetilde \Gamma = p^{-1}(\Gamma)\) be the inverse image of \(\Gamma\). We can lift \(\alpha\) to a cocycle \(\widetilde \alpha \colon \widetilde \Gamma \times \widetilde G \to \Delta\) by setting \(\widetilde \alpha(g,x) = \alpha(p(g), p(x))\). Now it is easily checked that the hypothesis of Theorem~\ref{thm : superrigidity} are satisfied.  In particular, \(\widetilde \Gamma\) has property \(\mathrm{(T)}\) --- note that this property does not hold whenever \(n=4\). Since \(\pi_{1}(\widetilde G)=\{e\}\) is trivial,
 the only group homomorphism \(\pi_{1}(\widetilde G) \to \Delta\) is obviously the trivial one.  Moreover, since the covering map \(p\) is open and surjective,  it follows that \(\widetilde \Gamma < \widetilde G\) is dense.  Therefore, by Theorem~\ref{thm : superrigidity} there are a group homomrphism \(\varphi\colon \widetilde\Gamma \to \Delta\) and a Borel map \(B\colon \widetilde G \to \Delta\) such that
\[
\varphi(g) = B(g\cdot x) \alpha (g, x)  B(x)^{-1}
\]
for all \(g\in \widetilde \Gamma\) and almost every \(x\in \widetilde G\).

Recall that \(\widetilde G = \Spin_n\) whose center \(Z(\widetilde G)\) is finite. (E.g., see~\cite[page~40]{Var04}) Let \(K = \ker \varphi\).  Because of Proposition~\ref{prop : covering}, \(\widetilde{\Gamma}\) is \(S\)-arithmetic in \(\widetilde G\) for \(S=\{\infty, 5\}\). Moreover, isogenous groups have the same \(S\)-rank,  so \(\rank_S \widetilde G \geq 2\).
(This observation follows from~\cite[Corollary~1.4.6~(a), Chapter~1]{Mar} and the fact that \(\rank_{\mathbb{Q}_p}\SO_n = \lfloor \frac{n}{2}\rfloor\) when \(p\equiv 1\) mod \(4\).)
By Margulis' normal subgroup theorem for \(S\)-arithmetic groups (see Theorem~\ref{thm : nst}) it follows that either:
\begin{enumerate-(1)}
\item
\([\widetilde \Gamma : K]<\infty\), or else
\item
\label{case : 2}
\(\varphi\colon \widetilde \Gamma \to \Delta\) is a \emph{virtual embedding}; i.e., \(K\) is finite.
\end{enumerate-(1)}

Now we analyze the two cases separately.
In case \([\widetilde \Gamma : K]<\infty\), then
there is a \(K\)-invariant measurable \(Z\subseteq \widetilde G\) with \(\widetilde m_{G} (Z)>0\) such that \(K\) acts ergodically on \((Z, \mu_{Z})\), where \(\mu_{Z}(A) = \widetilde m_{G}(A)/\widetilde m_{G}(Z)\). Then consider the cocycle
\[
\widetilde \alpha \restriction (K\times Z)\colon K\times Z \to \Delta.
\]

Next, for all \(x\in X\)  we can define the ``adjusted'' Borel homomorphism by \(f'(x) = B(x) \cdot f(x)\) so that \(\varphi\) is the cocycle associated to \(f'\restriction Z\),  and \(f'\restriction Z\) is a weak Borel reduction from \(\mathcal{R}(K\curvearrowright Z)\) to \(\mathcal{R}^*_m\).
It follows that \(f'(k\cdot x)=\varphi(k)\cdot f'(x) = f'(x)\) for all \(k\in K\) and \(x\in Z\). This shows that \(f'\colon Z \to Y\) is a \(K\)-invariant map. Since \(K\) acts ergodically on \((Z, \mu_{Z})\), it follows that \(f'\) is constant on a measure \(1\) set, which is a contradiction.

On the other hand,  the following lemma shows that there is no virtual embedding \(\varphi\colon \widetilde \Gamma \to \SO_m(\mathbb{Q})\). So, we can exclude case~\ref{case : 2}.

\begin{lemma}
\label{lem : Margulis}
If  \(\rho\colon \widetilde \Gamma \to \SO_m(\mathbb{Q})\) is
group homomorphism, then \(\rho(\widetilde \Gamma)\) is finite.
\end{lemma}

\begin{proof}
Since \(\widetilde \Gamma\) is finitely generated, there are only finitely many prime numbers \(p_1,\dotsc, p_n\) that appear in the denominators in the entries of the matrices in \(\rho(\widetilde{\Gamma})\).  So,
it suffices to show that the power of each prime appearing in the denominators of the matrix entries in  \(\rho(\gamma)\in \rho(\widetilde\Gamma)\) is uniformly bounded over \(\gamma \in \widetilde\Gamma\).
This will yield some \(N\in \mathbb{N}\) so that
\[
\rho(\widetilde \Gamma) \subseteq \big\{(a_{ij}) \in \SO_m(\Q) \mid a_{ij} \text{ has denominator less than }N\big \}.
\]
It will follow that \(\rho(\widetilde\Gamma)\) is discrete in \(\SO_n(\mathbb{R})\),  therefore \(\rho(\widetilde\Gamma)\) is finite because \(\SO_n(\mathbb{R})\) is compact.

For every \(q\in \{p_1,\dotsc,p_n\}\), let \(\SO_m(\Q)\to \SO_m(\Q_q)\) be the canonical embedding so that we can view \(\rho\) as a map \(\rho_q\colon \widetilde\Gamma\to \SO_m(\Q_q)\). 

If \(q \neq 5\), then let \(\rho_q\colon \widetilde \Gamma \to \SO_m(\mathbb{Q}_q)\).
Clearly, \(\rho_q(\widetilde \Gamma)\) need not be Zariski closed in \(\SO_m(\Q_q)\) so we cannot apply Theorem~\ref{Margulis : superrigidity} directly. 
However, the Zariski closure of \(\rho_q(\widetilde \Gamma)\) is an algebraic group over \(\mathbb{Q}_p\),  so let
\(M_q\) be the Zariski closure of \(\rho_q(\widetilde \Gamma)\).
Here notice that \(M_q\) need not be semisimple,  so we consider the semisimple group defined as \(L_p = M_q/ R(M_q)\) where \(R(M_q)\) is the \emph{(solvable) radical} of \(M_q\),  and is defined to be the maximal normal connected solvable subgroup of \(M_q\).
Clearly \(L_q\) is a semisimple  \(\mathbb{Q}_q\)-group and let  \(\sigma\colon M_q \to L_q\) be the quotient map.

Now the map \(\sigma\circ\rho_q\) is a group homomorphism with Zariski dense image.  Therefore, if we let \(K_q = \sigma\circ \rho_q( \widetilde \Gamma)\) we have that \(K_q\) is relatively compact  in \(L_q\) by Theorem~\ref{Margulis : superrigidity}.
It follows that \(\sigma^{-1}(\overline{K_q})\) is an extension of a compact group by a solvable group, so it is amenable.  Therefore, since \(\widetilde \Gamma\) has property \((\mathrm{T})\), we have that \(\rho_q(\widetilde \Gamma)\) is relatively compact because of Proposition~\ref{prop : (T) to amenable}.

Otherwise, if \(q = 5\), the statement follows by Theorem~\ref{Margulis : superrigidity}--\ref{superrigidity : (b)}, the fact that \(Z(\SO_m)\) is finite,  and the fact that every \(\mathbb{Q}_5\)-homomorphism
\(\widetilde G\to \SO_m\) is trivial because \(m<n\) and \(\Spin_n\) is almost \(\mathbb{Q}_5\)-simple. 
\end{proof}

This concludes the proof of Theorem~\ref{thm : irreducibility free}.

\section{The non-free part of \(\SO_n(\mathbb{Q})\curvearrowright \mathbb{S}^{n-1}\)}

In this section we show that whenever \(3\leq m<n\) the range of every Borel reduction from \(\mathcal{R}_n\) to \(\mathcal{R}_m\) is almost contained in the free part.  In fact,  we prove the following stronger statement.

\begin{proposition}
\label{lem : nonfree}
Suppose that \(3\leq m < n\) and that \(f\colon \mathbb{S}^{n-1}\to \mathbb{S}^{m-1}\) is a weak Borel reduction from \(\mathcal{R}_n\) to \(\mathcal{R}_m\). Then, there is a Borel \(\SO_n(\mathbb{Q})\)-invariant \(Y\subseteq \mathbb{S}^{n-1}\) with \(\mu(Y)=1\) such that \(f(Y)\) is contained in the free part. 
\end{proposition}

To prove Proposition~\ref{lem : nonfree} we use ideas from Coskey~\cite{Cos10} and the unpublished work of Thomas~\cite{Tho02unp}. Moreover, we will make use of the following cocycle reduction theorem, whose proof is discussed in Section~\ref{Section : added proof}.

\begin{theorem}
\label{thm : Thomas}
Let \(n \geq 5\) and let \(p\) be a prime number such that \(p\equiv 1 \mod 4\). 
Let \(X\) be a standard Borel \(\SO_n\big(\Z\big[\frac{1}{p}\big]\big)\)-space with
an invariant ergodic probability measure. 
Suppose that \(G\) is an algebraic \(\mathbb{Q}\)-group
such that \(\dim G < \frac{n(n-1)}{2}\) and that \(H \leq G(\mathbb{Q})\). Then for every Borel cocycle
\(\alpha \colon \SO_n\big(\Z\big[\frac{1}{p}\big]\big) \times X \to H\), there exists a cohomologous cocycle \(\beta\) such that \(\beta \big( \SO_n\big(\Z\big[\frac{1}{p}\big]\big)\times X\big)\) is contained in a finite subgroup of \(H\).
\end{theorem}

Before discussing the proof of Proposition~\ref{lem : nonfree} notice that if \(x \in \mathbb{S}^{m-1}\setminus \Fr \mathbb{S}^{m-1}\),  then there exists a nontrivial element \(M \in \SO_m(\mathbb{Q})\) such that \(M\cdot x = x\).  Namely,  \(x\) belongs to an eigenspace of \(M\). Since \(M\) has rational entries, we can easily find a basis for that eigenspace, whose vectors have rational coordinates.  This motivates the following definition. Denote by \(\overline{\mathbb{Q}}\) the algebraic closure of the fields of rational numbers.

\begin{definition}
The subspace \(V \subseteq \mathbb{R}^m\) is said to be a \emph{\(\overline{\mathbb{Q}}\)-based subspace} if and only if
there exists a (possibly empty) collection of vectors \(w_1, \dotsc, w_t \in (\overline{\mathbb{Q}}\cap\mathbb{R})^d \) such that
\(V = \spn \{w_1,\dotsc,w_t\}\). 
\end{definition}

\begin{proposition}
\begin{enumerate}
\item
	The intersection of \(\overline{\mathbb{Q}}\)-based subspace of \(\mathbb{R}^m\) is a \(\overline{\mathbb{Q}}\)-based subspace. 
\item
	For each \(y\in \mathbb{S}^{m-1}\),  there exists a unique minimal \(\overline{\mathbb{Q}}\)-based subspace \(V_y\) such that \(y\in V_y\). 
\end{enumerate}
\end{proposition}

\begin{proof}
\begin{enumerate}
\item
Suppose that \(V, W\) are \(\overline{Q}\)-based linear subspaces of \(\mathbb{R}^m\) with bases \(v_1,\dotsc, v_\ell \in \) and \(w_1,\dotsc, w_k\), respectively.
Let \(A\) and \(B\)
be matrices whose columns are the basis vectors \(v_1,\dotsc, v_{\ell}\)  and \(w_1,\dotsc, w_k\).
Then,  let \(\big(A\mid B\big)\) be the augmented matrix whose columns are the ones of \(A\) and \(B\).
  Since any element \(v\in V\cap W\) must satisfy \(v = A x = B y \) for some \(x\in \mathbb{R}^\ell\) and \(y\in \mathbb{R}^k\),  it follows that \(\big(A\mid B\big) \binom{x}{-y} = 0\). Thus we can deduce a basis for \(V\cap W\) from any basis for the nullspace of \((A\mid B)\). Since \((A\mid B)\) have entries in \(\overline{Q}\), we can easily find a basis for \(V\cap W\) in \((\overline{Q}\cap \mathbb{R})^m\).
\item
If \(V, W\) are minimal subspace containing \(y\), then \(y\in V\cap W\) and this will contradict the minimality of \(V\) and \(W\) unless \(V=W\).\qedhere
\end{enumerate}
\end{proof}

\begin{proof}[Proof of Proposition~\ref{lem : nonfree}]

Let \(f\colon \mathbb{S}^{n-1}\to \mathbb{S}^{m-1}\) be a weak Borel reduction from \(\mathcal{R}_n\) to \(\mathcal{R}_m\).  Let \(\mu = \mu_n\) be the spherical measure on \(\mathbb{S}^{n-1}\).  Put \(Y\coloneqq f^{-1}(\Fr \mathbb{S}^{m-1})\) and assume that \(\mu(Y)<1\), towards contradiction.
Clearly \(Y\)  is a \(\SO_n(\Q)\)-invariant Borel subset of \(\mathbb{S}^{n-1}\).  Since  
the action  \(\SO_n(\Q)\curvearrowright (\mathbb{S}^{n-1},\mu)\) is ergodic,  \(Y\) must be either null or conull, therefore \(\mu(Y)=0\) by the assumption of contradiction.

First, suppose that \(m=3\).
Let \(X_1 = \mathbb{S}^{n-1}\smallsetminus Y\). It follows that \(\mu(X_1) = 1\) and \(f(X_1)\) is contained in \(\mathbb{S}^2\smallsetminus \Fr(\mathbb{S}^2)\), which is the non-free part of the action \(\SO_3(\Q)\curvearrowright \Fr(\mathbb{S}^2)\).  Since every nontrivial \(A\in \SO_3(\mathbb{Q})\)  has a rotation axis,  it is clear that \(A\)  fixes exactly two points of \(\mathbb{S}^{2}\), the poles of the rotation axis,  thus \(f(X_1)\) is a countable set.  But then, since \(f\) is countable-to-one and \(\mu\) is non-atomic,  we have \(\mu(X_1)=\mu(f^{-1}(f(X_1))=0\),  a contradiction.

Next, suppose that \(m\geq 4\) and consequently \(n\geq 5\). Then put \(\Gamma = \SO_n\big(\mathbb{Z}\big[ \frac{1}{p}\big]\big) \) for some prime number \(p\equiv 1 \mod 4\), and note that \(f\) is a weak Borel reduction from \(\mathcal{R}(\Gamma\curvearrowright X)\) to \(\mathcal{R}_m\). Also, we assume without loss of generality that \(Y = \mathbb{S}^{n-1}\),  or equivalently \(f(\mathbb{S}^{n-1})\subseteq \mathbb{S}^{m-1}\setminus \Fr \mathbb{S}^{m-1}\).
Let \(F(\mathbb{R}^m)\) be the standard Borel space of linear vector space of \(\mathbb{R}^m\) with the Effros Borel structure.
Now,  for each \(x \in \mathbb{S}^{n-1}\), let \(V_x \in F(\mathbb{R}^m)\) be the unique minimal \(\overline{\mathbb{Q}}\)-based subspace such that \(f(x) \in V_x\); and consider the Borel map \( \mathbb{S}^{n-1} \to F(\mathbb{R}^m),  x\mapsto V_x\).
Since there are only countably many possibilities for \(V_x\),  there exists a Borel subset \(X_0 \subseteq \mathbb{S}^{n-1}\) with \(\mu(X_0) > 0\) and a fixed \(\overline{\mathbb{Q}}\)-based subspace \(V\) such that \(V_x = V\) for all \(x \in X_0\).  Let \(X_1 = \Gamma \cdot X_0\) be the saturation of \(X_0\).    Since \(\mu\) is ergodic and \(X_1\) is \(\Gamma\)-invariant,  it follows that \(\mu(X_1) = 1\). 
Clearly, \(V_x\) need not be constant for all \(x\in X_1\), however we can define
a Borel function \(c\colon X_1 \to X_1\) such that \(c(x) \in \Gamma \cdot x \cap X_0\) for all \(x \in X_1\).  Then,  for every \(x\in X_1\),  note that \(f\circ c(x)\) is equivalent to \(f(x)\), so \(f\circ c\) is also a weak Borel reduction from \(\mathcal{R}(\Gamma\curvearrowright X_1)\) to \(\mathcal{R}_m\).
Then, by replacing \(f\) with \(f' =f\circ c\),  we can assume that \(V_x = V\) for all \(x\in X_1\). 

Denote by \(\mathbb{P}(V)\) the standard Borel space of one-dimensional vector subspace of \(V\).
Consider the map \(\varphi\colon X_1\to \mathbb{P}(V), x\mapsto \spn\{f(x)\}\).
Let \(\PSO_m (\mathbb{Q})_{\{V\}}\) be  the setwise stabilizer of \(V\) in \(\PSO_m (\mathbb{Q})\).  (Recall that \(\PSO_n= \SO_n/  Z(\SO_n)\) where \(Z(\SO_n)=\{\pm I_n\}\) for \(n\) even and \(Z(\SO_n)=\{I_n\}\) for \(n\) odd.) To simplify our notation  we put \(Z=Z(\SO_m)\). Also, put \(G= \PSO_m(\mathbb{Q})\) and let \(H\) be the group of
projective linear transformations induced on \(V\) by  \(\PSO_m (\mathbb{Q})_{\{V\}}\). It follows that

\[
\dim H\leq \dim G =\frac{m(m-1)}{2}<\frac{n(n-1)}{2}.
\]

We consider the obvious action of \(H\) on \(\mathbb{P}(V)\).  That is, for any coset \(h = gZ \in H\) and \(v\in V\) we let \(h	\cdot \spn\{ v\} = \spn\{g\cdot v\}\).

\begin{lemma}
The map \(\varphi\) is a countable-to-one Borel homomorphism from \(\mathcal{R}(\Gamma\curvearrowright X_1)\) to \(\mathcal{R}(H\curvearrowright \mathbb{P}(V))\).
\end{lemma}
\begin{proof}
It is easily checked that \(\varphi\) is Borel and countable-to-one. To see that \(\varphi\) is a homomorphism
suppose that \(x,y\in X_1\) and \(\Gamma\cdot x = \Gamma \cdot y\). Then there exists \(g\in \SO_m(\mathbb{Q})\) such that \(g\cdot f(x)=f(y)\).  If \(h= gZ\in \PSO_m(\mathbb{Q})\),  we have \(h\cdot \varphi(x)=\varphi(y)\). To see that \(h\in H\),  note that \(hV\) is also a \(\mathbb{Q}\)-based vector space.  Therefore, \(f(y)\in V\cap hV\), which implies that \(hV = V\) because of minimality of \(V\).
\end{proof}

Now we consider the restriction of the action of \(H\) on
\(Y_1= H\cdot \varphi(X_1)\), i.e., the saturation of the image of \(\varphi\).

\begin{lemma}
The action of \(H\) on \(Y_1\)  is free.
\end{lemma}
\begin{proof}
If \(h=gZ\in H\) and \(h\cdot \spn\{x\} = \spn\{ x\}\),  then \(x\) is contained in the eigenspace \(W\) of \(g\) corresponding to some eigenvalue \(\lambda \in \mathbb{Q}\). By minimality of \(V\), we must have that \(V\subseteq W\) and so \(h(v)=\lambda v\) for all \(v\in V\).  Since \(\lambda = \pm 1\),  it follows that \(h\) is the identity in \(\PGL(V)\), therefore \(h\in 1_H\).
\end{proof}

Since \(H\) acts on \(Y_1\) freely,  we can define a cocycle \(\alpha\colon \Gamma \times X_1 \to H\) corresponding to the homomorphism \(\varphi\).  By Theorem~\ref{thm : Thomas},  the cocycle \(\alpha\) must be cohomologous to a cocycle \(\Gamma \times X_1 \to H\) with finite image.  Using the ergodicity of \(\SO_n\big(\Z\big[\frac{1}{p}\big]\big)\curvearrowright X_1\) and Lemma~\ref{lem : fin_subg},  it follows that
a measure \(1\) subset of \(X_1\) is mapped into a single equivalence class.
Then, since \(\mu\) is nonatomic,
 uncountably many elments of \(X_1\) are mapped into the same equivalence class, a contradiction.
\end{proof}

\begin{proof}[Proof of Theorem~\ref{thm : main}]
For \(m=2\),  Theorem~\ref{thm : main} follows easily from Theorem~\ref{thm : not treeable} and the fact that \(\mathcal{R}_2\) is hyperfinite and \(\mathcal{R}_n\) is not hyperfinite for \(n>m\). For \(m>2\), we combine the results of our last two sections. By Proposition~\ref{lem : nonfree} any Borel reduction from \(\mathcal{R}_n\) to \(\mathcal{R}_m\) is almost contained in the free part.  This leads to a contradiction arguing as in the proof of Theorem~\ref{thm : irreducibility free}.
 \end{proof}

\section{A cocycle reduction theorem}
\label{Section : added proof}

In this section we discuss the following cocycle reduction result for \(S\)-arithmetic groups.  It is the natural generalization of a result of Thomas \cite[Theorem~2.3]{Tho02} to the context of \(S\)-arithmetic groups.

\begin{theorem}
\label{thm : Thomas++}
Let \(S\) be a finite set of prime numbers. Let \(G\) be a connected semisimple \(\mathbb{Q}\)-group containing no connected normal \(\mathbb{Q}_p\)-subgroup of \(\mathbb{Q}_p\)-rank \(1\) for all \(p\in S\).
Suppose moreover that \(S\)-\(\rank G \geq 2\).
Denote by \(G_S =\prod_{p\in S\cup\{\infty\}} G(\mathbb{Q}_p)\).
Suppose that \(X\) is a standard Borel \(\Gamma\)-space with an invariant ergodic probability measure. 
Suppose that \(H\) is an algebraic \(\mathbb{Q}\)-group
such that \(\dim H < \dim G\) and that \(\Lambda < H(\mathbb{Q})\). Then for every Borel cocycle
\(\alpha \colon \Gamma \times X \to \Lambda\), there exists an equivalent cocycle \(\beta\) such that \(\beta(\Gamma\times X)\) is contained in a finite subgroup of \(\Lambda\).
\end{theorem}

\begin{example}
\begin{enumerate}
\item
Whenever \(S=\emptyset\) and \(G=\SL_n\) for \(n\geq 3\), then \(\Gamma_S = \SL_n(\mathbb{Z})\) is an arithemtic group and we can recover \cite[Theorem~2.3]{Tho02}.
\item
When \(G=\SO_n\) for \(n\geq 5\) and \(S=\{p\}\) for some prime \(p\equiv 1 \mod 4\) we recover Theorem~\ref{thm : Thomas}. Note that \(\rank_{\mathbb{R}} \SO_n = 0\) because \(\SO_n(\mathbb{R})\) is compact and \(\rank_{\mathbb{Q}_p} \SO_n = \lfloor 
\frac{n}{2}
\rfloor\geq 2\).
\end{enumerate}
\end{example}

As for \cite[Theorem~2.3]{Tho02},  we can prove Theorem~\ref{thm : Thomas++} with the cocycle techniques introduced by Adams and Kechris~\cite{AdaKec}. We sketch a proof below.

Let \((X,\mu)\)be a standard Borel \(G\)-space with invariant measure \(\mu\).
A \emph{finite extension} for the action \(G\curvearrowright (X,\mu)\) is a standard Borel \(G\)-space \(\hat X\) together with and invariant measure \(\hat \mu\) and a finite-to-one map \(\pi\colon \hat X \to X\) such that:
\begin{enumerate-(i)}
\item
For all \(g \in G\) and for all \( x \in \hat X\), we have \(\pi(g \cdot x)= g\cdot \pi(x)\), 
\item
\(\pi_*(\mu)=\mu\).
\end{enumerate-(i)}

\begin{theorem}
\label{lem : Thomas++}
Let \(G_S\) be as in the statement of Theorem~\ref{thm : Thomas++}.
Let \(X\) be a standard Borel \(G_S\)-space with an invariant ergodic probability measure.  Let \(k\) be a local field and let \(T\) be a simple algebraic \(k\)-group such that \(\dim T < \dim G\). Then for every Borel cocycle \(\alpha\colon G_S \times X \to T(k)\), there exists an equivalent cocycle \(\alpha'\) such that \(\alpha'(G_S\times X)\) is contained in a compact subgroup of \(T(k)\).
\end{theorem}
\begin{proof}[Proof]
Let \(\widetilde G\) be the universal cover of \(G\) and \(\pi\colon \widetilde{G}\to G\) be the covering map. Clearly, \(\widetilde G\) is simply connected.  Arguing as in~\cite[Corollary~5.4]{Mar} we can prove that
\(\widetilde G_S= \prod_{p\in S\cup\{\infty\}} \widetilde G(\mathbb{Q}_p)\) has property \(\mathrm{(T)}\).

We lift \(\alpha\) to a cocycle \(\widetilde\alpha\colon \widetilde G_S\times X \to T(k)\) defined by \(\widetilde\alpha(g, x)=\alpha(\pi(g),x)\). In view of \cite[Proposition~2.7]{AdaKec} it suffices to show that \(\widetilde\alpha\) is equivalent to a cocycle taking value in a compact subgroup of  \(T(k)\).
Unfortunately, we cannot apply Zimmer's cocycle superrigidity theorem for \(S\)-arithmetic groups~\cite[Theorem~10.1.6]{Zim} directly because \(\widetilde \alpha\) may be equivalent to a cocycle taking value into a proper subgroup of \(T(k)\) of the form \(L(k)\) where \(L\) is a \(k\)-subgroup of \(T\). However we can argue by minimality as in \cite[Theorem~3.5]{AdaKec} and find a finite extension \(\hat X\) of \(X\), a connected \(k\)-subgroup \(H\subseteq T\) and a cocycle \(\beta\colon \widetilde G_S\times \hat X \to H(k)\subseteq T(k)\) such that \(\beta\) is cohomologous to the lift of \(\tilde \alpha\).  Note that \(H\) might not be semisimple, so, before applying \cite[Theorem~10.1.6]{Zim} let \(R\) be the solvable radical of \(H\),  then \(L= H/R\) is a connected semisimple \(k\)-group and the canonical projection \(\tau\colon H\to L\) is a \(k\)-epimorphism.

We can find connected \(k\)-simple \(k\)-groups \(L_1,\dotsc,L_d\) together with a canonical \(k\)-epimorphism \(\tau_*\colon L\to L_1\times\dotsb\times L_d\) with finite kernel.  For \(i=1,\dotsc,d\), denote by \(\tau_i\) the composition of \(\tau_*\) with the projection of \(L_1 \times \dotsb \times L_d\to L_i\).

\begin{center}
\begin{tikzcd}
H \arrow[r, "\tau"] & L \arrow[r,"\tau_*"] \arrow[rd, "\tau_i"]& L_1\times\dotsb\times L_d \arrow[d, ] \\
&& L_i
\end{tikzcd}
\end{center}

Next, let \(\beta_i\colon  G_S\times \hat X \to L_i(k)\) be the cocycle defined by \(\beta_i(g,x)=\tau_i(\tau(\beta(g,x)))\).  Arguing again by minimality as in \cite[Theorem~3.5]{AdaKec} we can assume that \(\beta_i\) is not taking value in
a proper \(k\)-subgroup of \(L_i\),
 so the hypothesis of \cite[Theorem~10.1.6]{Zim} are satisfied. We conclude that \(\beta_i(G_S\times \hat X)\) is contained in a compact subgroup \(K_i\) of \(L_i(k)\). Then \(\tau_*\circ \tau \circ \beta\) is equivalent to a cocycle taking values in \(K_1 \times \dotsb \times K_d\). 

By \cite[Proposition~2.4]{AdaKec},  there is \(\ell\in L_1\times \dotsb\times L_d\) and a cocycle \(\gamma\colon G_S\times \hat X \to L_1\times \dotsb\times L_l\) such that
\(K\coloneqq\gamma(G_S\times \hat X)\subseteq \ell(L_1\times \dotsb\times L_l)\ell^{-1}\) and
\(\gamma\) is cohomologous to \(\beta\).  Since \(K\) is compact, the group \(A=\tau^{-1}(\tau^{-1}_*(K))\) is ameanable and \(\beta\) is equivalent to a cocycle taking value in \(A\subseteq H(k)\).
Note that \(G_S\) is a Kazhdan group.  This follows from the assumption that
\(G\) contains no connected \(k_p\)-subgroup of \(k_p\)-rank \(1\) for \(p\in S\).
(e.g., see Margulis~\cite[Chapter~III, 5.4]{Mar}).  Then, it follows from \cite[Theorem~9.1.1]{Zim} that \(\beta\)
is equivalent to a cocycle taking values in a compact subgroup of \(H(k)\) as desired.
\end{proof}

\begin{proof}[Sketch of Theorem~\ref{thm : Thomas++}]
We can argue exactly as in the proof of \cite[Theorem~2.3]{Tho02} using Theorem~\ref{lem : Thomas++} in place of \cite[Theorem~6.1]{Tho02}. To this purpose it is essential to keep in mind that \(\Gamma_S\) is a lattice in \(G_S\), and for any standard Borel \(\Gamma\)-space \(X\), induces and action of \(G_S\) on \(Y=X\times (G_S/\Gamma_S)\).  Moreover, any  any strict cocycle \(\beta\colon \Gamma_S \times X \to H\) induces a corresponding cocycle \(\hat \beta\colon G_S \times Y \to H\).
\end{proof}

\section{Incomparable  CBERs via left-translation}

Adams and Kechris~\cite{AdaKec} found continuum many pairwise incomparable equivalence relations up to Borel reducibility. In particular, they analyzed the equivalence relations induced by the shift action of the groups \(\Gamma_S \coloneqq\SO_7\big( \Z\big[\frac{1}{S}\big] \big)\) on the free part of the space of the space \(2^{\Gamma_S}\) equipped with the product topology,  for different sets of prime numbers \(S\). Our proof of Theorem~\ref{thm : irreducibility free} implicitly gives a proof of another result of incomparability.

If \(S\) is a set of prime numbers,  then we let  \(\Gamma_{n,S} = \SO_{n}(\Z[\frac{1}{S}])\).  Whenever \(S=\{p\}\),  we write \(\Gamma_{n,p} = \Gamma_{n,\{p\}}\).  Also let \(\mathcal{R}_{n,S} = \mathcal{R}(\Gamma_{n,S}\curvearrowright \mathbb{S}^{n-1})\) and \(\mathcal{R}^*_{n,S}\) be the restriction of \(\mathcal{R}_{n,S}\) to the free part.

\begin{theorem}
For all \(n\geq 5\),  there are continuum many pair-wise inequivalent equivalence subrelation of \(\mathcal{R}_n\).
\end{theorem}

\begin{proof}
Let \(\mathbb{P} =\{p\in \mathbb{N} \mid p\text{ is prime and }p\equiv 1 \mod 4 \}\). Let \(\mathcal{A} = \{S_x : x\in \mathfrak{c}\}\) be an almost disjoint family of subsets of \(\mathbb{P}\) of size the continuum.  I.e.,  if \(x\neq y\), then \(S_x\cap S_y \) is finite.
For distinct \(S,T\in\mathcal{A}\) and \(n\geq 5\),  we can prove that \(\mathcal{R}_{n,S}\) is not Borel reducible to \(\mathcal{R}_{n,T}\). To see this let
\(p\in S\smallsetminus T\).  Then any Borel reduction \(f\colon \mathbb{S}^{n-1}\to \mathbb{S}^{n-1}\) from \(\mathcal{R}_{n,S}\) to \(\mathcal{R}_{n,T}\) can be regarded as a weak Borel reduction from \(\mathcal{R}_{n,p}\) to \(\mathcal{R}_{n,T}\).  We are going to show that such weak Borel reduction cannot exist.

By ergodicity the image of \(f\) is almost contained in the free part or in the nonfree part.
In the first case,  we can argue as in the proof of Theorem~\ref{thm : irreducibility free}.
So let \(\alpha\colon \Gamma_{n,p}\times \mathbb{S}^{n-1} \to \Gamma_T\) be the associated cocycle.  Then, \(\alpha\) lifts to a cocycle \(\tilde \alpha\colon \pi^{-1}(\Gamma_{n,p}) \times \Spin_n(\mathbb{R}) \to \Gamma_{n,T}\) where \(\pi \colon \Spin_n \to \SO_n\) is the \(2\)--\(1\) covering map. Then the proof continues exactly as the one of Theorem~\ref{thm : irreducibility free}.  Note that \(\pi^{-1}(\Gamma_{n,p})\) has property \(\mathrm{(T)}\) and as in Lemma~\ref{lem : Margulis} every homomorphism from \(\pi^{-1}(\Gamma_{n,p})\) to \(\Gamma_{n,T}\) has finite image.

Otherwise,  suppose that the image of \(f\) is almost contained in the non-free part. Note that if \(f(x)\) is not in the free part of the action \(\Gamma_{n,T}\curvearrowright \mathbb{S}^{n-1}\), then there is \(X_0\subseteq \mathbb{S}^{n-1}\) with \(\mu(X_0)>0\) and there is a matrix \(M\in \Gamma_{n,T}\smallsetminus \{I,-I_n\}\) such that \(Mf(x)=f(x)\) for all \(x\in X_0\). (Here we can exclude \(-I_n\) if \(n\) is even because \(-I_n\) has no fixed points on \(\mathbb{S}^{n-1}\).)
 Therefore,  for every \(x\in X_0\), the point \(f(x)\) is contained in some proper \(\overline{\mathbb{Q}}\)-based vector space \(V_x\subseteq \mathbb{R}^{n}\).  Then the proof closely follows the one of Proposition~\ref{lem : nonfree}.  Let \(G, H\) be defined as in that proof.
 Since \(\dim V_x<n\),  it follows that \(\dim H< \dim G\).  Then,  by the same argument we conclude that a set of full measure is mapped into a single \(\mathcal{R}_{n,T}\)-class.
\end{proof}

\end{document}